%% file: main.tex
\documentclass[final]{siamltex}
\usepackage{epsfig}
\usepackage{amsmath,bm}
\usepackage{epic}
\usepackage{comment}
\usepackage{amssymb}
\newtheorem{condition}[theorem]{Condition}

\newtheorem{remark}[theorem]{Remark}

\newcommand{\Abs}[1]{\left|#1\right|}
\newcommand{\abs}[1]{\lvert#1\rvert}

\newcommand{\innprd}[2]{\left< #1 , #2 \right>}

\def\ge{\geqslant}
\def\le{\leqslant}

\def\norm#1{\left|\!\left| #1 \right|\!\right|}
\def\nnorm#1{|\!| #1 |\!|}

\def\op#1{{\mathcal #1}}
\def\vect#1{{\bf #1}}
\def\mat#1{{\bf #1}}

\def\beq{\begin{eqnarray}}
\def\eeq{\end{eqnarray}}
\def\beqs{\begin{eqnarray*}}
\def\eeqs{\end{eqnarray*}}
\def\eqdef{\stackrel{\mathrm{def}}{=}}
\def\R{\mathbb{R}}

\def\N{\mathbb{N}}

\newtheorem{algorithm}[theorem]{Algorithm}

\title{Optimal-order preconditioners for linear systems arising in the semismooth Newton solution of a class
of control-constrained problems}

\author{Andrei Dr{\u{a}}g{\u{a}}nescu\thanks{Department 
    of Mathematics and Statistics, University of Maryland, Baltimore
    County, 1000~Hilltop Circle, Baltimore, Maryland 21250 ({\tt draga@umbc.edu}).
This material is based upon work supported by the 
    U.S. Department of Energy Office of Science, Office of Advanced Scientific Computing Research, Applied
Mathematics program under Award Number DE-SC0005455,
    and by the National Science Foundation under awards
    DMS-1016177 and DMS-0821311.}
  \and{Jyoti Saraswat\thanks{Department 
    of Mathematics and Physics, Thomas More College, Crestview Hills, Kentucky  41017 ({\tt saraswj@thomasmore.edu}).}}}

\begin{document}

\maketitle

\begin{abstract} 
In this article we present a new multigrid preconditioner for the linear systems arising in the semismooth Newton method
solution of certain control-constrained, quadratic distributed optimal control problems.
Using a piecewise constant discretization of the control space, each semismooth Newton iteration essentially
requires inverting a principal submatrix of the matrix entering the normal equations of the associated unconstrained 
optimal control problem, the rows (and columns) of the submatrix representing the constraints deemed inactive
at the current iteration. Previously developed multigrid preconditioners for the aforementioned submatrices were
based on constructing a sequence of conforming coarser spaces, and proved to be of suboptimal quality for the class of problems considered. 
Instead, the multigrid preconditioner introduced in this work uses non-conforming coarse spaces, and it is shown that, 
under reasonable  geometric assumptions on the constraints that are deemed inactive, the  preconditioner approximates
the inverse of the desired submatrix to optimal order. The preconditioner is tested numerically 
on a classical elliptic-constrained optimal 
control problem and further on a constrained image-deblurring problem.

\end{abstract}



\pagestyle{myheadings}
\thispagestyle{plain}
\markboth{A.~DR{\u A}G{\u A}NESCU AND J.~SARASWAT}{MULTIGRID PRECONDITIONING FOR 
SEMISMOOTH NEWTON METHODS}

\input{intro.tex}

\input{problemdescription.tex}
\input{multigrid.tex}
\input{numerics.tex}

\section*{Acknowledgment} 
The authors thank the anonymous referees for their insightful comments.

\appendix
\input{appendixdeblurr}
\label{sec:deblur_conditions}

\bibliography{ssnmrefs}
\bibliographystyle{siam}

\end{document}

%% file: intro.tex
\section{Introduction}
\label{sec:intro}
The goal of this work is to construct optimal order multigrid preconditioners for optimal
control problems of the type
\beq
\label{eq:contprob}
\min_{u\in \op{U}} \frac{1}{2}\nnorm{\op{K} u-y_d}^2 +\frac{\beta}{2}\nnorm{u}^2\ ,\ \ a\le u\le b,\ \ a.e.,
\eeq
where $\op{U}\eqdef L^2(\Omega)$ with $\Omega\subset \R^n$ a bounded domain, $y_d\in \op{U}$ is given, and $\op{K}:\op{U}\to\op{V}$ is a linear continuous operator
with $\op{V}\hookrightarrow \op{U}$ being a compactly embedded subspace of~$\op{U}$. The parameter $\beta>0$ is used to 
adjust the size of the regularization term $\nnorm{u}^2$.
Throughout this article 
$\nnorm{\cdot}$ denotes the $\op{U}$-norm or the operator-norm of a
bounded linear operator in $\mathfrak{L}(\op{U})$. The functions  \mbox{$a, b\in \op{U}$} defining the inequality 
constraints in~\eqref{eq:contprob} satisfy $a(x)< b(x)$ for all \mbox{$x\in \Omega$}. These problems arise in the
optimal control of partial differential equations (PDEs), 
case in which $\op{K}$ represents the solution operator of a PDE. For example, the classical PDE-constrained optimization problem
\beq
\label{eq:contprobPDE}
\left\{
\begin{array}{l}\vspace{5pt}
\min\ \ \frac{1}{2}\nnorm{y-y_d}^2 +\frac{\beta}{2}\nnorm{u}^2\\\vspace{5pt}
\mathrm{subject\ to:} \ -\Delta y = u\ \ \mathrm{in}\ \Omega, \ \ y=0\ \ \mathrm{on}\ \partial\Omega,\\
\hspace{2cm} a\le u\le b\ \ a.e.,
\end{array}
\right .
\eeq
reduces to~\eqref{eq:contprob}  when replacing $y=\op{K}u$ in the cost functional of~\eqref{eq:contprobPDE},  where 
$\op{K}=(-\Delta)^{-1}:\op{U}\to \op{V}=H^1_0(\Omega)$.
A related problem, discussed in~\cite{MR2429872}, addresses the question of time-reversal for 
parabolic equations, a problem that is ill-posed. In this example 
we set $\op{K}u = \op{S}(T)u$, where $t\mapsto \op{S}(t)u$ is the time-$t$ solution operator of a linear
parabolic PDE with initial value~$u$, and $T>0$ is a fixed time.  If the solution $u_{\min}$
of the inverse problem needs to satisfy certain inequality constraints, e.g., when $u_{\min}$, perhaps
representing the concentration of a substance, is required to have values in $[0,1]$, then it is essential to
impose these constraints explicitly in the formulation of the optimization problem, as shown in~\eqref{eq:contprob}.
For obvious reasons, in the PDE-constrained optimization literature~\eqref{eq:contprob} 
is referred to as the reduced problem. 
For other applications, such as image deblurring, $\op{K}$ 
can be an explicitly defined integral operator
$$\op{K}u\eqdef\int_{\Omega} k(\cdot,x)u(x)dx\ ,$$
with $k\in \op{U}\otimes \op{U}$; here $u$ is the original image and $y=\op{K}u$ is the blurred image. Thus, by solving~\eqref{eq:contprob}
we seek to reconstruct the image $u$ whose blurred version is a given~$y_d$, subject to additional box constraints. 

We give a few references to works on multigrid methods for solving~\eqref{eq:contprob} with no inequality constraints.
In this case~\eqref{eq:contprob} 
is equivalent to the Tikhonov regularization of the ill-posed problem~$\op{K} u= y_d$, which 
in turn reduces to the linear system
\beq
\label{eq:illposedprobreg}
(\op{K}^*\op{K} +\beta I) u= \op{K}^* y_d ,
\eeq
representing the regularized normal equations of~$\op{K} u= y_d$.
A significant literature~\cite{MR1151773, MR97k:65299, MR2001h:65069,  
MR1986801, MR2429872, MR2421947, Dra:Soa:mgstokes}, to mention just a few references, is devoted to 
multigrid methods for~\eqref{eq:illposedprobreg} or the unregularized ill-posed problem.
Moreover, when $\op{K}$ is the solution operator of a linear PDE, an  alternative strategy is
to solve directly the indefinite systems representing the Karush-Kuhn-Tucker (KKT) optimality conditions instead of the
reduced system, and many works~\cite{MR2160699, MR2491821, MR2831057, MR2863635, MR3069094} are concerned with
multigrid methods for PDE-constrained optimization problems in unreduced form.
A comprehensive discussion of the latter strategy is found in~\cite{MR2505585}.

The presence of bound-constraints in~\eqref{eq:contprob} brings additional challenges to the solution process, 
since the KKT optimality conditions form a complementarity system as opposed to a linear or a smooth nonlinear system.
As shown by Hinterm{\" u}ller et al.~\cite{MR1972219}, the KKT system can be reformulated as a semismooth nonlinear system
for which a superlinearly convergent Newton's method can be devised -- the semismooth Newton method (SSNM). Moreover,
with controls discretized using piecewise constant finite elements,  Hinterm{\" u}ller and Ulbrich~\cite{MR2085262} 
have shown that the SSNM converges in a mesh-independent number of iterations for problems like~\eqref{eq:contprob}, 
so it is a very efficient solution method in terms of number of optimization iterations. 
A comprehensive discussion of SSNMs can be found in~\cite{MR2839219}. However, as with Newton's method, each SSNM iteration
requires the solution of a linear system, and the efficiency of the SSNM depends on the availability of high quality
preconditioners for the linear systems involved. Naturally, the question of devising preconditioners for 
SSNMs has received a lot of attention
in recent years, especially in the context of optimal control problems constrained by PDEs, e.g, 
see~\cite{MR2740620, MR2891920, doi:10.1137/140975711}, where preconditioners are primarily
targeting the sparse and indefinite KKT systems arising in the solution process.
For problems formulated as~\eqref{eq:contprob}, the SSNM solution essentially requires inverting  
at each iteration a principal submatrix of the matrix representing a discrete version~$\op{H}_h$
of $\op{H}=(\op{K}^*\op{K} +\beta I)$, where $h$ denotes the mesh size.
The multigrid preconditioner developed by Dr{\u a}g{\u a}nescu and Dupont in~\cite{MR2429872} for the operator
$\op{H}_h$ arising in the unconstrained problem~\eqref{eq:illposedprobreg} is shown, under reasonable conditions, 
to be of optimal order with respect to the discretization: namely, if we denote by~$\op{S}_h$ the multigrid 
preconditioner (thought of as an approximation of $(\op{H}_h)^{-1}$),  then
\beq
\label{eq:optordineq}
1-C\frac{h^p}{\beta} \le \frac{\innprd{\op{S}_h u}{u}}{\innprd{(\op{H}_h)^{-1} u}{u}} \le 1+C\frac{h^p}{\beta}, \ \ \forall u\in \op{U}_h\setminus\{0\},
\eeq
where $p>0$ is the convergence order of the discretization and $\op{U}_h\subset \op{U}$ is the discrete control space;
for continuous piecewise linear discretizations we have $p=2$.  A natural extension of the ideas in~\cite{MR2429872}
led to the suboptimal multigrid preconditioner developed by Dr{\u a}g{\u a}nescu in~\cite{doi:10.1080/10556788.2013.854356} for
principal submatrices of $\op{H}_h$, where $p$ is shown to essentially be~$1/2$ for a piecewise linear discretization.
The key aspect of defining the multigrid preconditioners for principal submatrices of~$\op{H}_h$ is the definition 
of the coarse spaces. The natural domain of a principal submatrix of~$\op{H}_h$, thought as an operator,  is a subspace
of $\op{U}_h$. The multigrid preconditioner developed in~\cite{doi:10.1080/10556788.2013.854356} is based on 
constructing coarse spaces that are subspaces of~$\op{U}_h$, i.e., conforming coarse spaces. 
A similar strategy was used by Hoppe and Kornhuber~\cite{MR1276702} in devising multilevel methods for obstacle
problems. However, a visual inspection of the eigenvectors corresponding to the extreme joint eigenvalues of $(\op{H}_h)^{-1}$ and
$\op{S}_h$ in~\cite{doi:10.1080/10556788.2013.854356} suggests that the multigrid preconditioner is suboptimal
precisely because of  the conformity of the coarse spaces.
Thus in this work we defined a new multigrid preconditioner based on non-conforming coarse spaces. 
While the new construction is 
limited to piecewise constant approximations of the controls, we can show that, under reasonable conditions, the approximation order 
of the preconditioner in~\eqref{eq:optordineq} is $p=1$. It also turns out that the analysis of the new preconditioner is
quite different from the analysis in~\cite{doi:10.1080/10556788.2013.854356}; fortunately it is also simpler.

This article is organized as follows: in  Section~\ref{sec:problem} we give a formal description of  the problem and we briefly describe the SSNM 
to justify the necessity of preconditioning principal submatrices of $\op{H}_h$. Section~\ref{sec:two-grid} forms the core of the article; here
we introduce and analyze the two-grid preconditioner, the main result being Theorem~\ref{th:optordprec}. In Section~\ref{sec:multigrid}
we extend the two-grid results to multigrid; this section follows closely the analogue extension in~\cite{doi:10.1080/10556788.2013.854356} with
certain modifications required by the non-conforming coarse spaces. In Section~\ref{sec:numerics} we show numerical experiments
conducted on two test problems: the elliptic constrained problem~\eqref{eq:contprobPDE} and the box-constrained image deblurring problem. We formulate
a set of conclusions in Section~\ref{sec:conclusions}. We included in Appendix~\ref{sec:appendixdeblur} a convergence analysis of a Gaussian
blurring operator that may also be of independent interest.

%% file: problemdescription.tex
\section{Problem formulation}
\label{sec:problem}
To fix ideas we assume the compactly embedded space to be $\op{V}\eqdef H_0^1(\Omega)$ and $\Omega$
to be polygonal or polyhedral. We denote by $\nnorm{u}_1=\nnorm{u}_{H^1_0(\Omega)}$ and we use the convention
$\nnorm{u}_0 = \nnorm{u}$. We also define the $H^{-1}$-norm by
\beqs
\nnorm{u}_{-1}\eqdef \sup_{v\in \op{V}\setminus\{0\}}\frac{\abs{\innprd{u}{v}}}{\nnorm{v}_1}\ ,
\eeqs
where $\innprd{u}{v}$ denotes the $L^2(\Omega)$-inner product.
We focus on a discrete version of~\eqref{eq:contprob} obtained by
discretizing the continuous operator $\op{K}$ using piecewise constant finite elements,
as considered by Hinterm{\" u}ller and Ulbrich in~\cite{MR2085262}. Let  $(\op{T}_{j})_{j\in \N}$  be  a family 
of shape-regular, nested triangulations of $\Omega$
with  with $h_j=\max_{T\in \op{T}_{j}} \mathrm{diam}(T)$ being the mesh-size of the triangulation $\op{T}_{j}$, and assume that
\beq
\label{eq:refinement_assump}
f_{low}\le h_{j+1}/h_{j}\le f_{high}
\eeq
for some $0<f_{low}\le f_{high}<1$ independent of $j$; for example, if $n=2$ a uniform mesh refinement
leads to $f_{low}=f_{high} = 1/2$. 
Furthermore, let $\op{U}_j$ be the space of continuous piecewise
constant functions with respect to $\op{T}_{j}$. Since the triangulations are nested, we
have $\op{U}_j\subset \op{U}_{j+1}$ for all $j\in \N$. We assume that a family of discretizations
$\op{K}_j\in\mathfrak{L}(\op{U}_j, \op{V}_j),\ j\in \N$,  is given, 
where $\op{V}_j\subset L^2(\Omega)$ is a finite element space  with properties specified below, and 
$\op{K}_j$ represents a discrete version of $\op{K}$. Note that it is not assumed that $\op{V}_j$ be
a subspace of  $\op{V}$.
The discrete optimization problem under scrutiny~is
\beq
\label{eq:discprob}
\min_{u\in \op{U}_j} \op{J}^{\beta}_j(u)\eqdef \frac{1}{2}\nnorm{\op{K}_j u-y^{(j)}_{d}}^2 +\frac{\beta}{2}\nnorm{u}^2\ ,\ \ a^{(j)}\le u\le b^{(j)}\ \  a.e.,
\eeq
where $a^{(j)}, b^{(j)}\in \op{U}_j$ are discrete functions representing $a, b$, and $y^{(j)}_{d}=\mathrm{Proj}_{\op{V}_j} y_d$.
As in~\cite{Dra:Pet:ipm}, we assume that the operators $\op{K},~\op{K}_j$ satisfy the \emph{smoothed approximation condition} (SAC) appended by $L^2$-$L^{\infty}$ 
stability of $\op{K}_j$:
\begin{condition}[SAC]
\label{cond:condsmooth} There exists a constant $C_1$ depending on $\Omega, \op{T}_{0}$ and independent of $j$
 so that 
\begin{enumerate}
\item[{\bf [a]}] smoothing:
\begin{equation}
\label{cond:par_smooth}
\max\left(\nnorm{\op{K}^* u}_{m}, \nnorm{\op{K} u}_{m}\right) \le C_1 \norm{u},\ \ \forall u\in L^2(\Omega),\  m=0, 1\ ;
\end{equation}
\item[{\bf [b]}] smoothed approximation: for $j\in \N$
\begin{equation}
\label{cond:consist}
\nnorm{\op{K} u - \op{K}_j u}  \le C_1 h_j\norm{u},
\ \ \forall u\in \op{U}_j;
\end{equation}
\item[{\bf [c]}] $L^2$-$L^{\infty}$ stability: for $j\in \N$
\begin{equation}
\label{cond:l2linf}
\nnorm{\op{K}_j u}_{L^{\infty}(\Omega)}  \le C_1 \norm{u},
\ \ \forall u\in \op{U}_j.
\end{equation}
\end{enumerate}
\end{condition}
\begin{remark}
A simple consequence of Condition~$\ref{cond:condsmooth}$ is that
\beq
\label{eq:h1l2est}
\nnorm{\op{K} u} \le C_1 \norm{u}_{-1}\ ,\ \ \forall u\in \op{U}\ .
\eeq
\end{remark}
\begin{proof}
Indeed, we have
$$
\nnorm{\op{K}u}^2=\innprd{u}{\op{K}^*\op{K} u} \le \nnorm{u}_{-1}\cdot \nnorm{\op{K}^*\op{K} u}_{1} \le
C_1 \nnorm{u}_{-1}\cdot  \nnorm{\op{K} u}\ ,
$$
and~\eqref{eq:h1l2est} follows.
\end{proof}

To describe the semismooth Newton method for the discrete optimization problem, we rewrite~\eqref{eq:discprob} in
vector form. Let $N_j=\mathrm{dim}(\op{U}_j)$ and $\varphi_1^{(j)},\dots, \varphi_{N_j}^{(j)}$ 
be the standard piecewise constant finite element basis in $\op{T}_{h_j}$.
First denote by $\vect{K}_j$ the matrix representing the operator $\op{K}_j$, and let $\vect{M}_j$ 
be the mass matrix in $\op{U}_j$, and $\widetilde{\vect{M}}_j$ the mass matrix in $\op{V}_j$. Note that the mass matrices $\vect{M}_j$
are diagonal. Then~\eqref{eq:discprob} is equivalent to
\beq
\label{eq:vectprob}
\min_{\vect{u}\in \R^{N_j}} J_j^{\beta}(\vect{u})\eqdef \frac{1}{2}\abs{\vect{K}_j \vect{u}-\vect{y}^{(j)}_{d}}^2_{\widetilde{\vect{M}}_j}
+\frac{\beta}{2}\abs{\vect{u}}^2_{\vect{M}_j}\ ,\ \ \vect{a}^{(j)}\le \vect{u}\le \vect{b}^{(j)},
\eeq
where $\vect{y}^{(j)}_{d}$ is the vector representing $y^{(j)}_{d}$, and $\abs{\vect{u}}_{\vect{M}} \eqdef \sqrt{\vect{u}^T \vect{M}\vect{u}}$.
To simplify the exposition  we omit the sub- and superscripts $j$ for the remainder of this section, so that
$\vect{K}=\vect{K}_j$, $N=N_j$, $\vect{a}=\vect{a}^{(j)}$,
etc. 
Similarly to~\eqref{eq:contprob}, the discrete optimization problem~\eqref{eq:vectprob} has a unique solution, which  satisfies
the KKT system
\beq
\label{eq:KKT}
\left\{
\begin{array}{l}
(\vect{K}^*\vect{K}+\beta\vect{I})\vect{u} +{\bm \lambda}_a-{\bm \lambda}_b = \vect{f}\\
\vect{a}-\vect{u}\le \vect{0},\ {{\bm \lambda}}_a\ge 0,\ (\vect{a}-\vect{u})\cdot {{\bm \lambda}}_a = \vect{0}\ \\
\vect{u}-\vect{b}\le \vect{0},\ {{\bm \lambda}}_b\ge 0,\ (\vect{u}-\vect{b})\cdot {{\bm \lambda}}_b = \vect{0}\ ,\\
\end{array}
\right.
\eeq
where $\vect{K}^*=\vect{M}^{-1}\vect{K}^T\widetilde{\vect{M}}$ is the adjoint of $\vect{K}$ with respect to the $L^2$-inner
product, $\vect{f}=\vect{K}^*\vect{y}_d$, and ${\bm \lambda}_a, {\bm \lambda}_b\in \R^N$ are the Lagrange multipliers. The inequalities
$\vect{u}\le \vect{v}$ and the vector-valued product $\vect{u}\cdot \vect{v}$ are to be understood componentwise.
The fact that we are able to write the KKT system in the form~\eqref{eq:KKT} is not completely obvious, and it relies 
on  the mass matrices $\vect{M}_j$ being diagonal. It is worth noting that
in~\cite{Dra:Pet:ipm, doi:10.1080/10556788.2013.854356}, where the controls were discretized using continuous piecewise linear functions, 
the mass-matrices were intentionally modified (equivalently  to using a quadrature for computing $L^2$-inner products) so that they be diagonal.
Following~\cite{MR1972219} (see also~\cite{doi:10.1080/10556788.2013.854356}), the complementarity problem~\eqref{eq:KKT} can be written 
as the non-smooth nonlinear system
\beq
\label{eq:sscomp}
\left\{
\begin{array}{l}
(\vect{K}^*\vect{K}+\beta\vect{I})\vect{u} -{\bm \lambda} = \vect{f}\\
{\bm \lambda}-\max(\vect{0}, \vect{K}^*\vect{K}\vect{u}-\vect{f}+\beta\vect{a}) - \min(\vect{0}, \vect{K}^*\vect{K}\vect{u}-\vect{f}+\beta\vect{b})=\vect{0}\ ,
\end{array}
\right.
\eeq
where ${\bm \lambda}=\bm{\lambda}_a-\bm{\lambda}_b$, $\bm{\lambda}_a=\max(\bm{\lambda},\vect{0}), \bm{\lambda}_b=-\min(\bm{\lambda},\vect{0})$.
We leave it as an exercise to verify that~\eqref{eq:KKT} is equivalent to~\eqref{eq:sscomp}.
Given the solution $(\vect{u},\bm{\lambda})$ of~\eqref{eq:KKT},  the following 
sets play a role in understanding the equivalence of~\eqref{eq:KKT} and~\eqref{eq:sscomp}:
\beqs
\label{eq:actinact}
\op{I}&=&\{i\in \{1,\dots,N\}\ :\ {\bm \lambda}_i+\beta(\vect{a}_i -\vect{u}_i) < 0\ \ \mathrm{and}\ \  {\bm \lambda}_i+\beta(\vect{b}_i -\vect{u}_i)>0\}\\
\op{A}^a&=&\{i\in \{1,\dots,N\}\ :\ {\bm \lambda}_i+\beta(\vect{a}_i -\vect{u}_i) \ge 0\ \ \mathrm{and}\ \  {\bm \lambda}_i+\beta(\vect{b}_i -\vect{u}_i)>0\}\\
\op{A}^b&=&\{i\in \{1,\dots,N\}\ :\ {\bm \lambda}_i+\beta(\vect{a}_i -\vect{u}_i) < 0\ \ \mathrm{and}\ \  {\bm \lambda}_i+\beta(\vect{b}_i -\vect{u}_i)\le 0\}\ .
\eeqs
Assume~\eqref{eq:sscomp} holds. If $i\in \op{I}$, the second equality  in~\eqref{eq:sscomp} implies that
${\bm \lambda}_i=0$; hence $\vect{a}_i < \vect{u}_i<\vect{b}_i$, so the constraint
corresponding to the $i^{\mathrm{th}}$ component of $\vect{u}$ is inactive. Instead, if $i\in \op{A}^a$, then $\vect{u}_i=\vect{a}_i$, so the lower constraints
are active; similarly, if $i\in \op{A}^b$, then $\vect{u}_i=\vect{b}_i$, so the upper constraints are active. So, if $(\vect{u}, {\bm \lambda})$
is the solution of~\eqref{eq:sscomp}, then $\op{I}$ is the set
of indices where the constraints are inactive, $\op{A}^a$ is the set of indices where the lower constraints are active, while
$\op{A}^b$ is the set of indices where the upper constraints are active. 

The system~\eqref{eq:sscomp} is in fact semismooth due to the fact that the
function $\vect{u}\mapsto \max(\vect{u},\vect{0})$ (from $\R^N$ to $\R^N$) is slantly differentiable~\cite{MR1972219}. 
Consequently,~\eqref{eq:sscomp} can be solved efficiently 
using the SSNM, which is  equivalent to the primal-dual active set method described below, as shown in~\cite{MR1972219}. 
The equivalence of the two is used to prove that the convergence is superlinear.
The SSNM is an iterative process that attempts to identify the sets $\op{I},~\op{A}^a,~\op{A}^b$ where the inequality constraints
are active/inactive. More precisely, at the $k^{\mathrm{th}}$ iteration,  
given sets $\op{I}_{k},~\op{A}^a_{k},~\op{A}^b_{k}$ partitioning $\{1,\dots,N\}$, we solve the system
\beq
\left\{
\begin{array}{l}\vspace{5pt}
  (\vect{K}^*\vect{K} + \beta \vect{I})\vect{u}^{(k+1)}-{\bm \lambda}^{(k+1)} =\vect{f}\ ,\\\vspace{5pt}
  \vect{u}^{(k+1)}_i=\vect{a}_i,\ \  \mathrm{for}\ \ i\in \op{A}^a_k,\ \ \  \vect{u}^{(k+1)}_i=\vect{b}_i,\ \  \mathrm{for}\ \ i\in \op{A}^b_k\ , \\\vspace{5pt}
  {{\bm \lambda}}^{(k+1)}_i=0,\ \  \mathrm{for}\ \ i\in \op{I}_k\ .
\end{array}
\label{eq:ssnit}
\right .
\eeq
The solution $(\vect{u}^{(k+1)}, {\bm \lambda}^{(k+1)})$ is then used to define the new sets
\beqs
\op{I}_{k+1}&=&\{i\ :\ 
{\bm \lambda}^{(k+1)}_i+\beta(\vect{a}_i -\vect{u}^{(k+1)}_i) < 0\ \ \mathrm{and}\ \  {\bm \lambda}^{(k+1)}_i+\beta(\vect{b}_i -\vect{u}^{(k+1)}_i)>0\}\ ,\\
\op{A}^a_{k+1}&=&\{i\ :\ 
{\bm \lambda}^{(k+1)}_i+\beta(\vect{a}_i -\vect{u}^{(k+1)}_i) \ge 0\ \ \mathrm{and}\ \  {\bm \lambda}^{(k+1)}_i+\beta(\vect{b}_i -\vect{u}^{(k+1)}_i)>0\}\ ,\\
\op{A}^b_{k+1}&=&\{i\ :\ 
{\bm \lambda}^{(k+1)}_i+\beta(\vect{a}_i -\vect{u}^{(k+1)}_i) < 0\ \ \mathrm{and}\ \  {\bm \lambda}^{(k+1)}_i+\beta(\vect{b}_i -\vect{u}^{(k+1)}_i)\le 0\}\ .
\eeqs
The key  problem in~\eqref{eq:ssnit} is to identify $\vect{u}^{(k+1)}_i$ for $i\in\op{I}_k$. 
If we denote the Hessian of~$J^{\beta}$ from~\eqref{eq:vectprob} by
$$\vect{H}=\vect{K}^*\vect{K} + \beta \vect{I}$$ and partition the matrix 
$\vect{H}$ according to the sets $\op{I}_{k}$ and $\op{A}_k=\op{A}^a_{k}\cup\op{A}^b_{k}$
$$
\vect{H}^{(k)}_{\mathrm{II}}=\vect{H}(\op{I}_{k},\op{I}_{k}),\ \vect{H}^{(k)}_{\mathrm{IA}}=\vect{H}(\op{I}_{k},\op{A}_{k}),\ 
\vect{H}^{(k)}_{\mathrm{AI}}=\vect{H}(\op{A}_{k},\op{I}_{k}),\ \vect{H}^{(k)}_{\mathrm{AA}}=\vect{H}(\op{A}_{k},\op{A}_{k}),\ 
$$
(we used M{\footnotesize ATLAB} notation to describe submatrices) and the vectors $\vect{u}^{(k+1)}$ and $f^{(k)}$~as 
$$\
\vect{u}^{(k+1)}_{\mathrm{I}}=\vect{u}(\op{I}_k),\ \vect{u}^{(k+1)}_{\mathrm{A}}=\vect{u}(\op{A}_k),\ 
\vect{f}^{(k)}_{\mathrm{I}}=\vect{f}(\op{I}_k),\ \vect{f}^{(k)}_{\mathrm{A}}=\vect{f}(\op{A}_k),
$$ 
then $\vect{u}^{(k+1)}_{\mathrm{A}}$ is given explicitly in~\eqref{eq:ssnit}, $\bm{\lambda}^{(k+1)}_{\mathrm{I}}=\vect{0}$, and $\vect{u}^{(k+1)}_{\mathrm{I}}$ satisfies
\beq
\label{eq:subsyst}
\vect{H}^{(k)}_{\mathrm{II}} \vect{u}^{(k+1)}_{\mathrm{I}} = \vect{f}^{(k)}_{\mathrm{I}} - \vect{H}^{(k)}_{\mathrm{IA}} \vect{u}^{(k+1)}_{\mathrm{A}}\ .
\eeq
The remaining components of $\bm{\lambda}^{(k+1)}$ are given by
$$
\bm{\lambda}^{(k+1)}_{\mathrm{A}}= \vect{H}^{(k)}_{\mathrm{AI}} \vect{u}^{(k+1)}_{\mathrm{I}}+\vect{H}^{(k)}_{\mathrm{AA}} \vect{u}^{(k+1)}_{\mathrm{A}}- \vect{f}^{(k)}_{\mathrm{A}}\ .
$$
Therefore, the main challenge in solving~\eqref{eq:ssnit} (which is a linear system) is in fact solving~\eqref{eq:subsyst}.
The goal of this work is to construct and analyze multigrid preconditioners for the matrices $\vect{H}^{(k)}_{\mathrm{II}}$ 
appearing in~\eqref{eq:subsyst}.

\section{The two-grid preconditioner}
\label{sec:two-grid}
As in~\cite{Dra:Pet:ipm, doi:10.1080/10556788.2013.854356}, we start by designing a two-grid preconditioner
for the principal submatrices of the Hessian $\vect{H}_j=(\vect{K}_j^*\vect{K}_j + \beta \vect{I})$
arising in the SSNM solution process of~\eqref{eq:vectprob}, then we follow the idea
in~\cite{MR2429872} to extend in Section~\ref{sec:multigrid} the two-grid preconditioner to a multigrid preconditioner of similar asymptotic quality.
In this section we assume that we  are solving~\eqref{eq:vectprob} at a fixed level $j$, and that we reached a certain iterate $k$ 
in the SSNM process, with a current guess at the inactive set
given by $\op{I}^{(j)}=\op{I}^{(j)}_k$. Since we are not changing the SSNM iteration we discard the sub- and superscripts $k$, and we
refer to $\op{I}^{(j)}_k$ as the ``current inactive set''.

For constructing the preconditioner it is preferable to regard the matrix $\vect{H}_j$ as a discretization
of the operator  
\beqs
\op{H}_j=(\op{K}_j^*\op{K}_j+\beta I)\in \mathfrak{L}(\op{U}_j)
\eeqs
representing the Hessian of $\op{J}^{\beta}_j$ in~\eqref{eq:discprob}.
We first define the (current) {\bf  inactive space}
$$
\op{U}_j^{\mathrm{I}} \eqdef \mathrm{span}\{\varphi^{(j)}_i\ :\ i\in \op{I}^{(j)}\} \subseteq \op{U}_j\ .
$$
Furthermore, denote by $\pi^{\mathrm{I}}_j$ the $L^2$-projection onto $\op{U}^{\mathrm{I}}_j$ and by 
$$
\Omega_j^{\mathrm{I}} = \bigcup_{i\in \op{I}^{(j)}} \mathrm{supp} (\varphi^{(j)}_i) \subset \Omega
$$
the {\bf inactive domain}. The matrix $\vect{H}^{\mathrm{I}}_j$ represents the
operator
\beq
\label{eq:inacthessdef}
\op{H}^{\mathrm{I}}_j \eqdef \pi^{\mathrm{I}}_j (\op{K}_j^*\op{K}_j+\beta I)\op{E}_j^{\mathrm{I}} \in \mathfrak{L}(\op{U}_j^{\mathrm{I}}),
\eeq
called here the {\bf inactive Hessian}, where $\op{E}_j^{\mathrm{I}}:\op{U}^{\mathrm{I}}_j\to \op{U}_j$ is the extension operator. Thus, our goal is to construct a two-grid preconditioner for
$\op{H}^{\mathrm{I}}_j$.

The first step, and perhaps the most notable achievement in this work,
is the construction of an appropriate coarse space: {\bf we define 
the coarse inactive space as the span of all coarse basis functions whose support 
intersect $\Omega_j^{\mathrm{I}}$ nontrivially}, i.e.,
\beqs
\op{U}_{j-1}^{\mathrm{I}} \eqdef \mathrm{span}\{\varphi^{(j-1)}_i\ :\ i\in \op{I}^{(j-1)}\} \subseteq \op{U}_{j-1}\ ,
\eeqs
with
\beq
\label{eq:coarseixdef}
\op{I}^{(j-1)} \eqdef \left\{i\in\{1,\dots, N_{j-1}\}\  :\ \mu(\mathrm{supp}(\varphi_i^{(j-1)})\cap \Omega_j^{\mathrm{I}})>0\right\}\ ,
\eeq
where $\mu$ is the Lebesgue measure in $\R^n$. Similarly, we define the coarse inactive domain by
\beq
\Omega_{j-1}^{\mathrm{I}} = \bigcup_{i\in \op{I}^{(j-1)}} \mathrm{supp} (\varphi^{(j-1)}_i) \subset \Omega\ .
\eeq

A few remarks are in order. First, since $\op{T}_j$ is a refinement of $\op{T}_{j-1}$, it follows that the set
$(\mathrm{supp}(\varphi_i^{(j-1)})\cap \Omega_j^{\mathrm{I}})$ is a (possibly empty) union of $\op{T}_j$-elements. Since each element
making up $\Omega_j^{\mathrm{I}}$ lies inside one element making up $\Omega_{j-1}^{\mathrm{I}}$, we have the inclusion
\beq
\label{eq:setinclusion}
\Omega_{j}^{\mathrm{I}}  \subseteq \Omega_{j-1}^{\mathrm{I}}\ .
\eeq
Second, we do not expect in general that $\op{U}_{j-1}^{\mathrm{I}} \subseteq  \op{U}_{j}^{\mathrm{I}}$. However, we have
\beq
\label{eq:equivdomaineqspacenest}
\op{U}_{j-1}^{\mathrm{I}} \subseteq  \op{U}_{j}^{\mathrm{I}}\ \ \ \ \mathrm{if\ and \ only\ if}\ \ \ \ \Omega_{j-1}^{\mathrm{I}}=\Omega_j^{\mathrm{I}}\ .
\eeq
We also denote
$$\partial_{\mathrm{n}}\Omega_j^{\mathrm{I}} \eqdef \Omega_{j-1}^{\mathrm{I}}\setminus \Omega_j^{\mathrm{I}}, $$
a set we call the numerical boundary of $\Omega_j^{\mathrm{I}}$ with respect to the coarse mesh. In Figures~\ref{fig:grid8} 
and~\ref{fig:grid32_64} we show the sets~$\Omega_j^{\mathrm{I}}$ (dark-gray) and $\partial_{\mathrm{n}}\Omega_j^{\mathrm{I}}$ 
(light-gray) for a few cases on uniform triangular grids on $\Omega=[0,1]\times[0,1]$.

We would like to contrast the definition~\eqref{eq:coarseixdef} of the coarse indices with that of Dr{\u a}g{\u a}nescu~\cite{doi:10.1080/10556788.2013.854356},
where a coarse basis function enters the span of the coarse inactive space if $\mathrm{supp}(\varphi_i^{(j-1)})\subseteq \Omega_j^{\mathrm{I}}$; this 
would define a coarse inactive space that lies inside the fine inactive space $\op{U}^{\mathrm{I}}_{j}$, 
and the inclusion~\eqref{eq:setinclusion} would be reversed, that is,
the coarse inactive domain would be included in the fine inactive domain.

\begin{figure}[!hb]
        \includegraphics[width=5.3in]{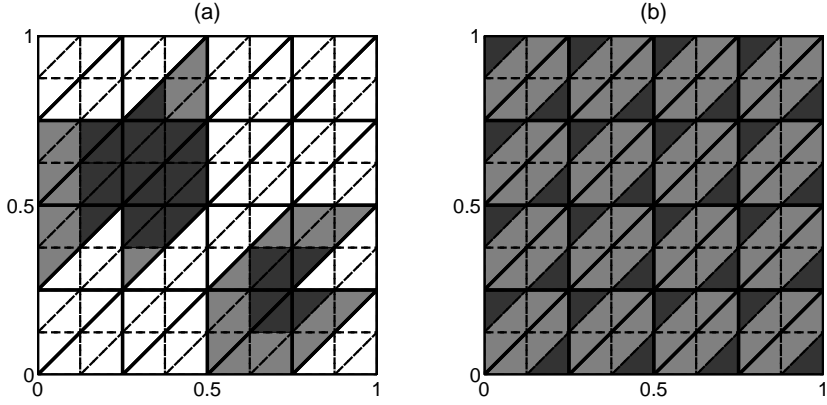}
\caption{In dark-gray we show $\Omega_j^{\mathrm{I}}$, and in light-gray we represent $\partial_{\mathrm{n}}\Omega_j^{\mathrm{I}}$ 
for $n=8$ on a uniform triangular grid. 
In {\textnormal{(a)}} $\Omega_j^{\mathrm{I}}$ is the best representation of a union of two disks on the current grid; 
in~\textnormal{(b)} $\Omega_j^{\mathrm{I}}$ is a
set for which $\Omega_{j-1}^{\mathrm{I}}=\Omega=[0,1]\times[0,1]$.}
\label{fig:grid8}
\end{figure}

\begin{figure}[!htb]
\begin{center}
        \includegraphics[width=5.3in]{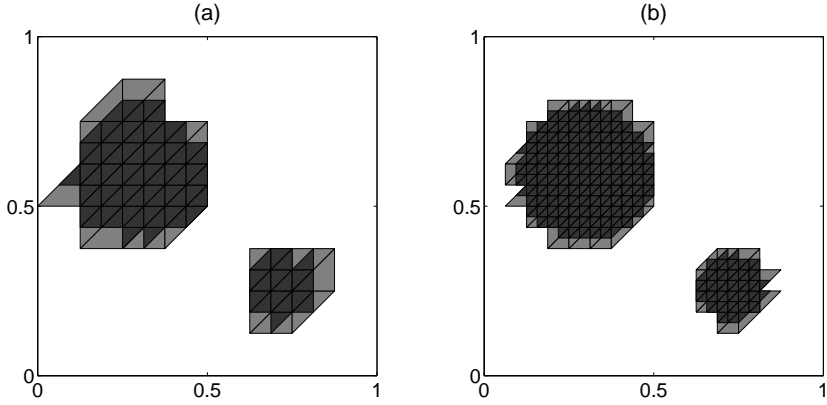}
\caption{In dark-gray we show $\Omega_j^{\mathrm{I}}$, and in light-gray we represent $\partial_{\mathrm{n}}\Omega_j^{\mathrm{I}}$ 
for $n=16$ {\textnormal{(a)}} and $n=32$ {\textnormal{(b)}} on a uniform triangular grid $($meshlines are omitted 
to enhance picture clarity$)$. For both cases,
$\Omega_j^{\mathrm{I}}$ is the best representation on the current grid of the same union of two disks 
used to generate $\Omega_j^{\mathrm{I}}$ in Figure~$\ref{fig:grid8}$~\textnormal{(a)}. The ratio of the areas of the numerical boundaries
in \textnormal{(b)} vs. \textnormal{(a)} is approximately~$0.64$.}
\label{fig:grid32_64}
\end{center}
\end{figure}

We now define the two-grid preconditioner $\op{M}_j\in \mathfrak{L}(\op{U}_j^{\mathrm{I}})$ by
\beq
\label{eq:twogridprec}
\op{M}_j \eqdef \pi^{\mathrm{I}}_j\left((\op{H}^{\mathrm{I}}_{j-1})^{-1}\pi^{\mathrm{I}}_{j-1}+\beta^{-1}(I- \pi^{\mathrm{I}}_{j-1})\right)\ .
\eeq
The definition~\eqref{eq:twogridprec} is rooted in the two-grid preconditioner definition from~\cite{MR2429872, doi:10.1080/10556788.2013.854356}; the
difference lies the presence of the action of the projection $\pi^{\mathrm{I}}_j$ as the last step in~\eqref{eq:twogridprec} (left-most term), 
which is necessary precisely because
$\op{U}_{j-1}^{\mathrm{I}}$ is not expected to be a subspace of $\op{U}_j^{\mathrm{I}}$.
An operator related to $\op{M}_j$, necessary for the analysis, 
is  
\beq
\label{eq:twogridprecinv}
\op{S}_j \eqdef \pi^{\mathrm{I}}_j\left(\op{H}^{\mathrm{I}}_{j-1}\pi^{\mathrm{I}}_{j-1}+\beta(I- \pi^{\mathrm{I}}_{j-1})\right)\ .
\eeq
\begin{remark}
\label{rem:symmetry}
Both $\op{M}_j$ and $\op{S}_j$ are symmetric with respect to the $L^2$-inner product, that is,
\beqs
\innprd{\op{M}_j u}{v} = \innprd{u}{\op{M}_j v},\ \ \ \innprd{\op{S}_j u}{v} = \innprd{u}{\op{S}_j v},\ \ \forall u, v\in \op{U}_{j}^{\mathrm{I}}\ .
\eeqs
In addition, if $\op{U}_{j-1}^{\mathrm{I}} \subseteq  \op{U}_{j}^{\mathrm{I}}$, then 
$\op{M}_j = (\op{S}_j)^{-1}$.
\end{remark}

The key to the last assertions in Remark~\ref{rem:symmetry} is that $\pi^{\mathrm{I}}_j$ has no effect (hence can be discarded) when
$\op{U}_{j-1}^{\mathrm{I}} \subseteq  \op{U}_{j}^{\mathrm{I}}$.
Our ultimate  goal is to estimate the spectral distance between $(\op{H}^{\mathrm{I}}_{j})^{-1}$ and $\op{M}_j$, as 
a measure of their spectral equivalence (see definition below). 
As an intermediate step we
will estimate the spectral distance between $\op{H}^{\mathrm{I}}_{j}$ and $\op{S}_j$.
 
Given a Hilbert space $(\op{X},\innprd{\cdot}{\cdot})$, 
we  denote by $\mathfrak{L}_+(\op{X})$ the set of symmetric positive definite operators in $\mathfrak{L}(\op{X})$.
The spectral distance between  $A, B\in \mathfrak{L}_+(\op{X})$, introduced in~\cite{MR2429872}
to analyze multigrid preconditioners for inverse problems like~\eqref{eq:illposedprobreg},
is given by
$$
d_{\op{X}}(A,B)=\sup_{u\in \op{X}\setminus\{0\}}\Abs{\ln\frac{\innprd{A u}{u}}{\innprd{B u}{u}}}\ .
$$
If $\delta$ is the smallest number for which the following inequalities hold
$$1-\delta\le \frac{\innprd{A u}{u}}{\innprd{B u}{u}} \le 1+ \delta, \ \ \forall u\ne 0,$$
and $\delta\ll1$, then $d_{\op{X}}(A,B) \approx \delta$. The spectral distance not only allows to write the above 
inequalities in a more compact form, but some of its properties (including the fact that it is a distance function) are also used in the analysis. 
The main result in this article is the following theorem.

\begin{theorem}
\label{th:optordprec}
Assuming Condition$~\ref{cond:condsmooth}$ holds, there exists constants $\delta>0$ and  $C_{\mathrm{tg}}>0$ independent of $j$ and 
the inactive set $\op{I}^{(j)}$ so that if 
$\beta^{-1}(h_j+\mu(\partial_{\mathrm{n}}\Omega_j^{\mathrm{I}}))<\delta$ the following holds:
\beq
\label{eq:optordprec}
d_{\op{U}_j^{\mathrm{I}}}(\op{M}_j,(\op{H}^{\mathrm{I}}_{j})^{-1})\le C_{\mathrm{tg}} \beta^{-1}\left(h_j + \mu(\partial_{\mathrm{n}}\Omega_j^{\mathrm{I}})\right)\ . 
\eeq
\end{theorem}
We postpone the proof of Theorem~\ref{th:optordprec} after a few preliminary results.
\begin{remark}
\label{rem:twogridhypopt}
Without further formalizing the argument, we would like to comment on the optimality of the result in
Theorem~$\ref{th:optordprec}$. First, it should be recognized that $\mu(\partial_{\mathrm{n}}\Omega_j^{\mathrm{I}})$ 
can be $O(1)$ for certain choices of  $\Omega^{\mathrm{I}}_j$. For example,
if $\op{T}_j$ is a uniform refinement of $\op{T}_{j-1}$ in two dimensions, and $\Omega^{\mathrm{I}}_j$ contains exactly one
level-$j$ subdivision of each of the level-$(j-1)$ triangles that make up $\Omega$, as shown in Figure~$\ref{fig:grid8}$ $($b$)$, 
then $\op{U}^{\mathrm{I}}_{j-1} = \op{U}_{j-1}$
$($the entire coarse space$)$ and $\Omega^{\mathrm{I}}_{j-1} = \Omega$; thus 
$\mu(\partial_{\mathrm{n}}\Omega_j^{\mathrm{I}}) = \frac{3}{4}\mu(\Omega)$. In this case the two-grid
preconditioner is not efficient. However, if $\Omega_j^{\mathrm{I}}$ is a good approximation of the 
correct inactive domain $\Omega^{\mathrm{I}} = \{x \in \Omega : a < u_{\min}(x) < b\}$, and $\Omega^{\mathrm{I}}$
is sufficiently regular, e.g., has a Lipschitz boundary, then we expect
$\mu(\partial_{\mathrm{n}}\Omega_j^{\mathrm{I}})\approx C h_j$. It is in this sense that
we regard Theorem~$\ref{th:optordprec}$ as proof of the fact that the two-grid preconditioner $\op{M}_j$ approximates
the operator $(\op{H}^{\mathrm{I}}_{j})^{-1}$ with optimal order. Figures~$\ref{fig:grid8}$~$($b$)$ and~$\ref{fig:grid32_64}$ $($a$)$ and $($b$)$ 
show a progression of $\partial_{\mathrm{n}}\Omega_j^{\mathrm{I}}$ $($in gray$)$
for the case when $\Omega_j^{\mathrm{I}}$ is a union of two discrete representations of disks  on  grids with $n=8, 16, 32$.
The ratio of the gray areas in Figure~$\ref{fig:grid32_64}$ $($a$)$ and $($b$)$ representing 
$\mu(\partial_{\mathrm{n}}\Omega_j^{\mathrm{I}})/\mu(\partial_{\mathrm{n}}\Omega_{j-1}^{\mathrm{I}})$ is $0.64$. Furthermore, this 
ratio converges to $1/2$ as the resolution tends to zero.
\end{remark}

The optimality result in the following lemma is a critical component for the proof of Theorem~\ref{th:optordprec}.
\begin{lemma}
\label{lemma:projapprox}
There exists a constant $C_2$ independent of the mesh-size $h_j$ and the inactive set
$\op{I}^{(j)}$ so that
\beq
\label{eq:projapprox}
\nnorm{(I-\pi_{j-1}^{\mathrm{I}})u}_{-1}\le C_2 h_j \nnorm{u}\ ,\ \ \ \forall u\in \op{U}_j^{\mathrm{I}}\ ,
\eeq
where $(I-\pi_{j-1}^{\mathrm{I}})u$ is extended with zero outside its support.
\end{lemma}
\begin{proof}
For $ u\in \op{U}_j^{\mathrm{I}}$ we have
\beqs
\nnorm{(I-\pi_{j-1}^{\mathrm{I}})u}_{-1} &=& \sup_{v\in H_0^1(\Omega)} \frac{\innprd{(I-\pi_{j-1}^{\mathrm{I}})u}{v}}{\nnorm{v}_1}.
\eeqs
Since $\mathrm{supp}((I-\pi_{j-1}^{\mathrm{I}})u)\subseteq \Omega_{j}^{\mathrm{I}}\cup \Omega_{j-1}^{\mathrm{I}}=\Omega_{j-1}^{\mathrm{I}}$, 
\beqs
\innprd{(I-\pi_{j-1}^{\mathrm{I}})u}{v} &=&\innprd{(I-\pi_{j-1}^{\mathrm{I}})u}{v-\pi_{j-1}^{\mathrm{I}}v}
 = \int_{\Omega_{j-1}^{\mathrm{I}}} ((I-\pi_{j-1}^{\mathrm{I}})u)(v-\pi_{j-1}^{\mathrm{I}}v) d\mu\\
&\le&
\nnorm{(I-\pi_{j-1}^{\mathrm{I}})u}_{L^2(\Omega_{j-1}^{\mathrm{I}})} \cdot \nnorm{(I-\pi_{j-1}^{\mathrm{I}})v}_{L^2(\Omega_{j-1}^{\mathrm{I}})} \\
&\le & 2 \nnorm{u} \cdot \nnorm{(I-\pi_{j-1}^{\mathrm{I}})v}_{L^2(\Omega_{j-1}^{\mathrm{I}})}.
\eeqs
Now
\beqs
\nnorm{(I-\pi_{j-1}^{\mathrm{I}})v}^2_{L^2(\Omega_{j-1}^{\mathrm{I}})} &=& \sum_{i\in\op{I}^{(j-1)}}\nnorm{(I-\pi_{j-1}^{\mathrm{I}})v}^2_{L^2(T_i^{(j-1)})}\\
&\le& \tilde{C}^2 h^2_{j-1} \sum_{i\in\op{I}^{(j-1)}}|v|^2_{H^1(T_i^{(j-1)})} \le f_{low}^{-2} \tilde{C}^2 h^2_j |v|^2_{H^1(\Omega)}\ ,
\eeqs
where $\tilde{C}$ is the constant (uniform with respect to $i$ and $j$ due to shape-regularity) appearing in the 
Bramble-Hilbert Lemma on each element $T_i^{(j-1)}$ in $\op{T}_{j-1}$ with\mbox{ $i\in \op{I}^{(j-1)}$}; 
we also used the fact that the 
$L^2$-projection is local for the finite element space under consideration, that is, 
$\pi_{j-1}^{\mathrm{I}}v|_{T_i^{(j-1)}}$ is the average of $v$ on the element $T_i^{(j-1)}$.
It follows that
\beqs
\innprd{(I-\pi_{j-1}^{\mathrm{I}})u}{v}& \le {f^{-1}_{low}} \tilde{C} h_j \nnorm{u}\cdot |v|_1,\ \ \forall v\in H_0^1(\Omega),
\eeqs
which implies the desired result with $C_2=f^{-1}_{low} \tilde{C}$.
\end{proof}

\begin{remark}
It is remarkable that the constant $C_2$ is independent of the inactive set, and depends only on the constant appearing in the Bramble-Hilbert lemma
and the refinement ratio.
Also, what makes the optimal estimate~\eqref{eq:projapprox} possible, is the inclusion $\Omega_{j}^{\mathrm{I}}\subseteq \Omega_{j-1}^{\mathrm{I}}$. If
our choice of spaces had led to $\Omega_{j-1}^{\mathrm{I}}\not\subseteq \Omega_{j}^{\mathrm{I}}$, then the term to estimate
would be $\nnorm{(I-\pi_{j-1}^{\mathrm{I}})v}_{L^2(\Omega_{j}^{\mathrm{I}})}$, which is expected to be of 
size $\sqrt{\mu(\Omega_{j}^{\mathrm{I}}\setminus \Omega_{j-1}^{\mathrm{I}})}$; as shown in~\textnormal{\cite{doi:10.1080/10556788.2013.854356}},
the latter term is often of size $\sqrt{h_j}$.
\end{remark}
\begin{proposition}
\label{lma:HSapprox}
Under the assumptions of Theorem$~\ref{th:optordprec}$ there exists $C_3, \delta>0$ independent of $j, \beta$
and the fine inactive set $\op{I}^{(j)}$ so that, if $h_j/\beta<\delta$, then 
\beq
\label{eq:HSapprox}
d_{\op{U}_j^{\mathrm{I}}}(\op{S}_j,\op{H}^{\mathrm{I}}_{j})\le C_3 \frac{h_j}{\beta}\ .
\eeq
\end{proposition}
\begin{proof}
As in Lemma~\ref{lemma:projapprox}, functions are extended with zero outside their support when necessary.
We have for $u\in \op{U}_j^{\mathrm{I}}$
\beqs
\lefteqn{\innprd{(\op{H}^{\mathrm{I}}_{j}-\op{S}_j)u}{u}}\\
&=&\innprd{\pi^{\mathrm{I}}_j \left(\op{K}_j^*\op{K}_j-
\pi^{\mathrm{I}}_{j-1}\op{K}_{j-1}^*\op{K}_{j-1}\pi^{\mathrm{I}}_{j-1}
\right)u}{u} 
= \innprd{\op{K}_ju}{\op{K}_ju}-
\innprd{\op{K}_{j-1}\pi^{\mathrm{I}}_{j-1}u}{ \op{K}_{j-1} \pi^{\mathrm{I}}_{j-1}u}\\
& =&
\underbrace{\nnorm{\op{K}_ju}^2 - \nnorm{\op{K} u}^2}_{A_1} +\underbrace{\nnorm{\op{K} u}^2 -\nnorm{\op{K} \pi^{\mathrm{I}}_{j-1} u}^2}_{A_2}
+ \underbrace{\nnorm{\op{K} \pi^{\mathrm{I}}_{j-1} u}^2 - \nnorm{\op{K}_{j-1} \pi^{\mathrm{I}}_{j-1} u}^2}_{A_3}\ .
\eeqs
Condition~\ref{cond:condsmooth} implies that
\beqs
\Abs{A_1}=\Abs{\nnorm{\op{K}_ju}^2 - \nnorm{\op{K} u}^2} & \le & \nnorm{(\op{K}_j-\op{K})u}\cdot \left(\nnorm{\op{K}_ju} + \nnorm{\op{K} u}\right)
\le 2 C_1^2 h_j  \nnorm{u}\ ,
\eeqs
and a similar estimate holds for the term $A_3$ with a constant depending on $C_1$ and $f_{low}$.
For the second term $A_2$ we have
\beqs
\Abs{A_2}&=&\Abs{\nnorm{\op{K} u}^2 -\nnorm{\op{K} \pi^{\mathrm{I}}_{j-1} u}^2} \le 
\nnorm{\op{K}(I-\pi^{\mathrm{I}}_{j-1}) u} \cdot \left(\nnorm{\op{K} u} +\nnorm{\op{K} \pi^{\mathrm{I}}_{j-1} u}\right)\ \\
&\stackrel{\eqref{eq:h1l2est},\eqref{eq:projapprox}}{\le}& 2 C_1^2 C_2\nnorm{u}\ .
\eeqs
The symmetry of $(\op{H}^{\mathrm{I}}_{j}-\op{S}_j)$ implies that
\beq
\nnorm{\op{H}^{\mathrm{I}}_{j}-\op{S}_j}\le C' h_j\ ,
\eeq
with $C'$ depending on $C_1, C_2, f_{low}$, but not on  $h_j$ or the inactive set $\op{I}^{(j)}$.
The rest of the argument follows closely the proof of Theorem~4.9 in~\cite{Dra:Pet:ipm}, and we provide it for completeness.
Since $\op{H}^{\mathrm{I}}_{j}$ is symmetric and $\innprd{\op{H}^{\mathrm{I}}_{j} u}{u}\ge \beta \nnorm{u}^2, \  \forall u\in \op{U}_j^{\mathrm{I}}$,
it follows that 
$$\sigma\left((\op{H}^{\mathrm{I}}_{j})^{-\frac{1}{2}}\right)\subseteq (0,\beta^{-\frac{1}{2}}]\ ,
\ \ \mathrm{therefore}\ \ \nnorm{(\op{H}^{\mathrm{I}}_{j})^{-\frac{1}{2}}}\le \beta^{-\frac{1}{2}}\ .$$
Hence
\beqs
\nnorm{I-(\op{H}^{\mathrm{I}}_{j})^{-\frac{1}{2}}\op{S}_j (\op{H}^{\mathrm{I}}_{j})^{-\frac{1}{2}}} \le \beta^{-1} \nnorm{\op{H}^{\mathrm{I}}_{j}-\op{S}_j} \le C'\frac{h_j}{\beta}\ .
\eeqs
If $C'h_j/\beta<1/2$, then 
$$W\left((\op{H}^{\mathrm{I}}_{j})^{-\frac{1}{2}}\op{S}_j (\op{H}^{\mathrm{I}}_{j})^{-\frac{1}{2}}\right) = 
\left\{\frac{\innprd{\op{S}_j u}{ u}}{\innprd{\op{H}^{\mathrm{I}}_{j}u}{u}}\ :\ u\in \op{U}_j^{\mathrm{I}}\right\}
\subseteq\left[\frac{1}{2},\frac{3}{2}\right]\ ,$$
where $W(\op{A})$ represent the numerical range of the operator $\op{A}$. 
By Lemma~3.2 in~\cite{MR2429872}
\beqs
\sup_{u\in \op{U}_j^{\mathrm{I}}}
\Abs{\ln\frac{\innprd{\op{S}_j u}{u}}{\innprd{\op{H}^{\mathrm{I}}_{j} u}{u}}}
\le \frac{3}{2}\nnorm{I-(\op{H}^{\mathrm{I}}_{j})^{-\frac{1}{2}}\op{S}_j (\op{H}^{\mathrm{I}}_{j})^{-\frac{1}{2}}}\le \frac{3 C'}{2}\frac{h_j}{\beta}\ ,
\eeqs
which proves~\eqref{eq:HSapprox} with $C_3=3 C'/2$.
\end{proof}

Another essential element in the proof of Theorem~\ref{th:optordprec}
is the following additional enriched level-$j$ inactive set and associated space:
\beq
\label{eq:inactbardef}
\widehat{\op{I}}^{(j)}&\eqdef& \{i\in \{1,\dots,N_j\}\ :\ \mathrm{supp}(\varphi^{(j)})\subseteq \Omega_{j-1}^{\mathrm{I}}\},\\
\label{eq:spacebardef}
\widehat{\op{U}}_j^{\mathrm{I}}&\eqdef& \mathrm{span}\{\varphi^{(j)}\in \op{U}_j\ : i\in \widehat{\op{I}}^{(j)}\}\ .
\eeq
It is obvious that $\widehat{\op{U}}_j^{\mathrm{I}}$ includes  both 
$\op{U}_{j-1}^{\mathrm{I}}$ and ${\op{U}_j^{\mathrm{I}}}$; it could also be regarded as the level-$j$ inactive space whose inactive domain
is identical to $\Omega_{j-1}^{\mathrm{I}}$. We should also point out that the coarse inactive index set generated
by $\widehat{\op{I}}^{(j)}$ is still ${\op{I}}^{(j-1)}$, therefore the coarse inactive space associated with $\widehat{\op{U}}_j^{\mathrm{I}}$
is identical to that associated with ${\op{U}}_j^{\mathrm{I}}$, namely $\op{U}_{j-1}^{\mathrm{I}}$. Let 
$\widehat{\pi}_j^{\mathrm{I}}$ be the $L^2$-projection onto $\widehat{\op{U}}_j^{\mathrm{I}}$ and 
$\widehat{\op{E}}_j^{\mathrm{I}}: \widehat{\op{U}}_j^{\mathrm{I}}\to \op{U}_j$ be the extension operator.
We now define the inactive Hessian and two-grid preconditioners associated with~$\widehat{\op{I}}^{(j)}$, all of which are to be regarded as operators in
$\mathfrak{L}(\widehat{\op{U}}_j^{\mathrm{I}})$:
\beq
\label{eq:barHdef}
\widehat{\op{H}}^{\mathrm{I}}_j& \eqdef & \widehat{\pi}^{\mathrm{I}}_j (\op{K}_j^*\op{K}_j+\beta I)\widehat{\op{E}}_j^{\mathrm{I}},\\
\label{eq:barSdef}
\widehat{\op{S}}_j &\ \eqdef\  &\op{H}^{\mathrm{I}}_{j-1}\pi^{\mathrm{I}}_{j-1}+\beta(I- \pi^{\mathrm{I}}_{j-1}),\\
\label{eq:barMdef}
\widehat{\op{M}}_j &\ \eqdef\  &(\op{H}^{\mathrm{I}}_{j-1})^{-1}\pi^{\mathrm{I}}_{j-1}+\beta^{-1}(I- \pi^{\mathrm{I}}_{j-1}).
\eeq
Also, let 
$\eta_j^{\mathrm{I}}: {\op{U}}_j^{\mathrm{I}}\to\widehat{\op{U}}_j^{\mathrm{I}}$ be the 
extension operator.

\begin{lemma}
\label{lma:hinv-htildeinv-comp}
There exists constants $\delta>0$ and  $C_4>0$ independent of $j$ and the inactive set $\widehat{\op{I}}^{(j)}$
so that if $\beta^{-1}\mu(\partial_{\mathrm{n}}\Omega_j^{\mathrm{I}})<\delta$ then
\beq
\label{eq::hinv-htildeinv-comp}
d_{\op{U}_j^{\mathrm{I}}}\left((\op{H}^{\mathrm{I}}_j)^{-1},{\pi}^{\mathrm{I}}_j(\widehat{\op{H}}^{\mathrm{I}}_j)^{-1}\eta_j^{\mathrm{I}}\right) 
\le C_4 \frac{\mu(\partial_{\mathrm{n}}\Omega_j^{\mathrm{I}})}{\beta}\ .
\eeq
\end{lemma}
\begin{proof}
The first task is to find a practical expression for the operator ${\pi}_j^{\mathrm{I}}(\widehat{\op{H}}^{\mathrm{I}}_j)^{-1}\eta_j^{\mathrm{I}}$.
Let ${\op{U}}^{\mathrm{I}}_{j,c}$ be the orthogonal complement of ${\op{U}}_j^{\mathrm{I}}$ in $\widehat{\op{U}}_j^{\mathrm{I}}$, so that
$\widehat{\op{U}}_j^{\mathrm{I}} = {\op{U}}_j^{\mathrm{I}}\oplus {\op{U}}^{\mathrm{I}}_{j,c}$; note that functions in ${\op{U}}^{\mathrm{I}}_{j,c}$
are supported in $\partial_{\mathrm{n}}\Omega_j^{\mathrm{I}}$. Furthermore, let ${\pi}^{\mathrm{I}}_{j,c}$ be the
orthogonal projector on ${\op{U}}^{\mathrm{I}}_{j,c}$ be the projector 
and $\eta^{\mathrm{I}}_{j,c}:{\op{U}}^{\mathrm{I}}_{j,c}\to \widehat{\op{U}}_j^{\mathrm{I}}$ be the extension operator.
Following the block-splitting of the matrix representing $\widehat{\op{H}}_j^{\mathrm{I}}$, we define the operators
\beqs
\op{H}^{\mathrm{I}}_{j,oc}&=&\pi_j^{\mathrm{I}}\widehat{\op{H}}_j^{\mathrm{I}}\eta^{\mathrm{I}}_{j,c}\in\mathfrak{L}({\op{U}}^{\mathrm{I}}_{j,c},{\op{U}}_j^{\mathrm{I}}), 
\ \op{H}^{\mathrm{I}}_{j, co}=\pi^{\mathrm{I}}_{j,c}\widehat{\op{H}}_j^{\mathrm{I}}\eta_j^{\mathrm{I}}\in\mathfrak{L}({\op{U}}^{\mathrm{I}}_j,{\op{U}}^{\mathrm{I}}_{j,c})\\
\op{H}^{\mathrm{I}}_{j,cc}&=&\pi^{\mathrm{I}}_{j,c}\widehat{\op{H}}_j^{\mathrm{I}}\eta^{\mathrm{I}}_{j,c}\in\mathfrak{L}({\op{U}}^{\mathrm{I}}_{j,c},{\op{U}}^{\mathrm{I}}_{j,c}).
\eeqs
Naturally, $\op{H}_j^{\mathrm{I}}=\pi_j^{\mathrm{I}}\widehat{\op{H}}^{\mathrm{I}}\eta_j^{\mathrm{I}}$, and ${\op{H}}_{j,oc}^{\mathrm{I}}=({\op{H}}_{j,co}^{\mathrm{I}})^*$.
To ease notation, in the first part of this analysis we eliminate the sub- or super-scripts ``$j$'', since the level does not vary, 
so $\widehat{\op{I}} = \widehat{\op{I}}^{(j)}$, $\widehat{\op{U}}^{\mathrm{I}} = \widehat{\op{U}}^{\mathrm{I}}_{j}$, 
${\op{H}}_{co}^{\mathrm{I}}={\op{H}}_{j,co}^{\mathrm{I}}$, etc. 
Accordingly, if $\widehat{u}=u+u_c$ with $u\in {\op{U}}^{\mathrm{I}}, u_c\in {\op{U}}^{\mathrm{I}}_c$, we have
\beq
\label{eq:hhataction}
\widehat{\op{H}}^{\mathrm{I}} \widehat{u}=\underbrace{({\op{H}}^{\mathrm{I}}u + {\op{H}}_{oc}^{\mathrm{I}}u_c)}_{\mathrm{in}\  {\op{U}}^{\mathrm{I}}} + 
\underbrace{({\op{H}}_{co}^{\mathrm{I}}u+ {\op{H}}^{\mathrm{I}}_{cc}u_c)}_{\mathrm{in}\  {\op{U}}^{\mathrm{I}}_c}\ .
\eeq
We also define the Schur-complement of ${\op{H}}^{\mathrm{I}}$ in  $\widehat{\op{H}}^{\mathrm{I}}$
$$
\op{G}={\op{H}}^{\mathrm{I}}_{cc}-{\op{H}}_{co}^{\mathrm{I}}(\op{H}^{\mathrm{I}})^{-1}{\op{H}}_{oc}^{\mathrm{I}}\ \in 
\mathfrak{L}({\op{U}}^{\mathrm{I}}_c,{\op{U}}^{\mathrm{I}}_c).
$$
Note that $\op{G}$ is symmetric. For $u_c\in {\op{U}}^{\mathrm{I}}_c$ define $u=-(\op{H}^{\mathrm{I}})^{-1}{\op{H}}_{oc}^{\mathrm{I}}u_c\in {\op{U}}^{\mathrm{I}}$
and $\widehat{u}= u + u_c$. A simple calculation shows that
\beq
\label{eq:coercG}
\innprd{\op{G}u_c}{u_c}& \stackrel{\eqref{eq:hhataction}}{=}  & \innprd{\widehat{\op{H}}^{\mathrm{I}} \widehat{u}}{\widehat{u}} \ge \beta\innprd{\widehat{u}}{\widehat{u}}=
\beta \left(\nnorm{u}^2+\nnorm{u_c}^2 \right) \ge \beta \nnorm{u_c}^2.
\eeq
(This is simply saying that the smallest eigenvalue of the Schur-complement is greater than the smallest eigenvalue of the
original operator). Hence, it follows that
\beq
\label{eq:Gnorm}
\nnorm{\op{G}^{-1}}\le \beta^{-1}\ .
\eeq
When solving $\widehat{\op{H}}^{\mathrm{I}} \widehat{u} = y$ for $y\in {\op{U}}^{\mathrm{I}}$,
standard block-elimination yields $\widehat{u}=u+u_c$ with 
$$
{u}=(\op{H}^{\mathrm{I}})^{-1}\left(I + {\op{H}}_{oc}^{\mathrm{I}}\op{G}^{-1}{\op{H}}_{co}^{\mathrm{I}} (\op{H}^{\mathrm{I}})^{-1}\right) y,\ \ \ 
u_c=-\op{G}^{-1}\op{H}_{co}^{\mathrm{I}} (\op{H}^{\mathrm{I}})^{-1}y\ .
$$
The first equation above shows that 
\beq
\label{eq:extHinv}
{\pi}^{\mathrm{I}}(\widehat{\op{H}}^{\mathrm{I}})^{-1}\eta^{\mathrm{I}} = 
(\op{H}^{\mathrm{I}})^{-1}\left(I + {\op{H}}_{oc}^{\mathrm{I}}\op{G}^{-1}{\op{H}}_{co}^{\mathrm{I}} (\op{H}^{\mathrm{I}})^{-1}\right).
\eeq

To estimate the spectral distance between ${\pi}^{\mathrm{I}}(\widehat{\op{H}}^{\mathrm{I}})^{-1}\eta^{\mathrm{I}}$ and $({\op{H}}^{\mathrm{I}})^{-1}$ we bound
\beq
\nonumber
\lefteqn{\sup_{u\in {\op{U}}^{\mathrm{I}}}\Abs{1-\frac{\innprd{{\pi}^{\mathrm{I}}(\widehat{\op{H}}^{\mathrm{I}})^{-1}\eta^{\mathrm{I}}u}{u}}{\innprd{({\op{H}}^{\mathrm{I}})^{-1}u}{u}}}}\\
\nonumber
&\stackrel{\eqref{eq:extHinv}}{=}&
\sup_{u\in {\op{U}}^{\mathrm{I}}\setminus\{0\}}\Abs{\frac{\innprd{(\op{H}^{\mathrm{I}})^{-1}{\op{H}}_{oc}^{\mathrm{I}}\op{G}^{-1}{\op{H}}_{co}^{\mathrm{I}} (\op{H}^{\mathrm{I}})^{-1} u}
{u}}{\innprd{({\op{H}}^{\mathrm{I}})^{-1}u}{u}}}\\
\nonumber
 &\stackrel{v=(\op{H}^{\mathrm{I}})^{-\frac{1}{2}}u}{=} &
\sup_{v\in {\op{U}}^{\mathrm{I}}\setminus\{0\}}\Abs{\frac{\innprd{(\op{H}^{\mathrm{I}})^{-\frac{1}{2}}{\op{H}}_{oc}^{\mathrm{I}}\op{G}^{-1}{\op{H}}_{co}^{\mathrm{I}} (\op{H}^{\mathrm{I}})^{-\frac{1}{2}} v}{v}}{\innprd{v}{v}}} = \nnorm{(\op{H}^{\mathrm{I}})^{-\frac{1}{2}}{\op{H}}_{oc}^{\mathrm{I}}\op{G}^{-1}{\op{H}}_{co}^{\mathrm{I}} (\op{H}^{\mathrm{I}})^{-\frac{1}{2}}}\\
\label{eq:firstest}
&\le & \nnorm{(\op{H}^{\mathrm{I}})^{-\frac{1}{2}}{\op{H}}_{oc}^{\mathrm{I}}}^2 \cdot \nnorm{\op{G}^{-1}}\le
\beta^{-1}\nnorm{(\op{H}^{\mathrm{I}})^{-\frac{1}{2}}{\op{H}}_{oc}^{\mathrm{I}}}^2 \ ,
\eeq
since $\nnorm{(\op{H}^{\mathrm{I}})^{-\frac{1}{2}}{\op{H}}_{oc}^{\mathrm{I}}} = \nnorm{{\op{H}}_{co}^{\mathrm{I}}(\op{H}^{\mathrm{I}})^{-\frac{1}{2}}}$ 
as they are dual to each other.

We resume using the index $j$ as we estimate $\nnorm{(\op{H}_j^{\mathrm{I}})^{-\frac{1}{2}}{\op{H}}_{j, oc}^{\mathrm{I}}}$.
Following~\eqref{eq:coercG}, for all $u_c\in {\op{U}}^{\mathrm{I}}_{j,c}$
\beqs
\nnorm{(\op{H}_j^{\mathrm{I}})^{-\frac{1}{2}}{\op{H}}_{j,oc}^{\mathrm{I}}u_c}^2
&=&\innprd{{\op{H}}_{j,co}^{\mathrm{I}}(\op{H}_j^{\mathrm{I}})^{-1}{\op{H}}_{j,oc}^{\mathrm{I}} u_c}{u_c} \stackrel{\eqref{eq:coercG}}{\le} 
\innprd{({\op{H}}^{\mathrm{I}}_{j,cc}-\beta I)u_c}{u_c}\\
&\le& \innprd{\pi^{\mathrm{I}}_{j,c}(\widehat{\op{H}}_j^{\mathrm{I}}-\beta I)\eta^{\mathrm{I}}_{j,c}u_c}{u_c}.
\eeqs
If we define by ${\op{E}}_{j,c}^{\mathrm{I}}\in\mathfrak{L}({\op{U}}^{\mathrm{I}}_{j,c},{\op{U}}_{j})$ the extension-with-zero operator, then
\beqs
\pi^{\mathrm{I}}_{j,c}(\widehat{\op{H}}_j^{\mathrm{I}}-\beta I)\eta^{\mathrm{I}}_{j,c} = 
\pi^{\mathrm{I}}_{j,c}(\op{K}_j^*\op{K}_j){\op{E}}_{j,c}^{\mathrm{I}}.
\eeqs
Therefore,
\beqs
\nnorm{(\op{H}_j^{\mathrm{I}})^{-\frac{1}{2}}{\op{H}}_{j,oc}^{\mathrm{I}}u_c}^2
& \le & \innprd{\pi^{\mathrm{I}}_{j,c}(\op{K}_j^*\op{K}_j){\op{E}}_{j,c}^{\mathrm{I}}u_c}{u_c} = \nnorm{\op{K}_j{\op{E}}_{j,c}^{\mathrm{I}}u_c}^2\\
& \le & \nnorm{\op{K}_j{\op{E}}_{j,c}^{\mathrm{I}}u_c}_{L^{\infty}(\partial_{\mathrm{n}}\Omega_j^{\mathrm{I}})}^2 \: \mu(\partial_{\mathrm{n}}\Omega_j^{\mathrm{I}})\\
&\stackrel{\eqref{cond:l2linf}}{\le}  &
C_1^2\nnorm{{\op{E}}_{j,c}^{\mathrm{I}}u_c}_{L^{2}(\Omega)}^2 \: \mu(\partial_{\mathrm{n}}\Omega_j^{\mathrm{I}})
=C_1^2\nnorm{u_c}^2 \: \mu(\partial_{\mathrm{n}}\Omega_j^{\mathrm{I}})\ .
\eeqs
It follows that
\beq
\label{eq:factest}
\nnorm{(\op{H}_j^{\mathrm{I}})^{-\frac{1}{2}}{\op{H}}_{j,oc}^{\mathrm{I}}}\le C_1\sqrt{\mu(\partial_{\mathrm{n}}\Omega_j^{\mathrm{I}})}\ .
\eeq
So, by~\eqref{eq:firstest} and~\eqref{eq:factest} we get
\beqs
\sup_{u\in {\op{U}}_j^{\mathrm{I}}}
\Abs{1-\frac{\innprd{{\pi}_j^{\mathrm{I}}(\widehat{\op{H}}_j^{\mathrm{I}})^{-1}\eta_j^{\mathrm{I}}u}{u}}{\innprd{({\op{H}}_j^{\mathrm{I}})^{-1}u}{u}}}
\le C_1^2\beta^{-1}\mu(\partial_{\mathrm{n}}\Omega_j^{\mathrm{I}})\ .
\eeqs
If $C_1^2\beta^{-1}\mu(\partial_{\mathrm{n}}\Omega_j^{\mathrm{I}}) < \frac{1}{2}$, by Lemma~3.2 in~\cite{MR2429872}
\beqs
\Abs{\ln\frac{\innprd{{\pi}_j^{\mathrm{I}}(\widehat{\op{H}}_j^{\mathrm{I}})^{-1}\eta_j^{\mathrm{I}}u}{u}}{\innprd{({\op{H}}_j^{\mathrm{I}})^{-1}u}{u}}}
\le \frac{3}{2}C_1^2\beta^{-1}\mu(\partial_{\mathrm{n}}\Omega_j^{\mathrm{I}})\ ,
\eeqs
which concludes the proof.
\end{proof}

We now return to the proof of Theorem~\ref{th:optordprec}.
\begin{proof}
We refer to the space $\widehat{\op{U}}_j^{\mathrm{I}}$ defined in~\eqref{eq:spacebardef} and the associated operators
$\widehat{\op{H}}^{\mathrm{I}}_j,  \widehat{\op{S}}_j, \widehat{\op{M}}_j$ defined in~\eqref{eq:barHdef}-\eqref{eq:barMdef}.
Cf. Remark~\ref{rem:symmetry}, because $\op{U}_{j-1}^{\mathrm{I}} \subseteq \widehat{\op{U}}_j^{\mathrm{I}}$, we have
$(\widehat{\op{S}}_j)^{-1} = \widehat{\op{M}}_j\ .$ By Lemma~3.10 in~\cite{MR2429872} we have
\beq
\label{eq:specdistinv}
d_{\widehat{\op{U}}_j^{\mathrm{I}}}\left(\widehat{\op{M}}_j, (\widehat{\op{H}}^{\mathrm{I}}_j)^{-1}\right)=
d_{\widehat{\op{U}}_j^{\mathrm{I}}}\left(\widehat{\op{S}}_j, \widehat{\op{H}}^{\mathrm{I}}_j\right)\ .
\eeq
Hence,
\beqs
\lefteqn{d_{{\op{U}}_j^{\mathrm{I}}}\left({\op{M}}_j, ({\op{H}}^{\mathrm{I}}_j)^{-1}\right) }\\
&=&\sup_{u\in {\op{U}}_j^{\mathrm{I}}\setminus \{0\}} 
\Abs{\ln\frac{\innprd{{\pi}^{\mathrm{I}}_j\left(\left(\op{H}^{\mathrm{I}}_{j-1}\right)^{-1}\pi^{\mathrm{I}}_{j-1}+\beta^{-1}(I- \pi^{\mathrm{I}}_{j-1})\right) u}{u}}{\innprd{({\op{H}}^{\mathrm{I}}_j)^{-1} u}{u}}}\\
 & \stackrel{\op{U}_{j}^{\mathrm{I}} \subseteq \widehat{\op{U}}_j^{\mathrm{I}}}{=} & 
\sup_{u\in {\op{U}}_j^{\mathrm{I}}\setminus \{0\}} 
\Abs{\ln\frac{\innprd{\widehat{\pi}^{\mathrm{I}}_j\left(\left(\op{H}^{\mathrm{I}}_{j-1}\right)^{-1}\pi^{\mathrm{I}}_{j-1}+\beta^{-1}(I- \pi^{\mathrm{I}}_{j-1})\right) u}{u}}{\innprd{({\op{H}}^{\mathrm{I}}_j)^{-1} u}{u}}}\\
&\le &
\sup_{u\in {\op{U}}_j^{\mathrm{I}}\setminus \{0\}} 
\Abs{\ln\frac{\innprd{\widehat{\op{M}}_j u}{u}}{\innprd{(\widehat{\op{H}}^{\mathrm{I}}_j)^{-1} u}{u}}}+
\sup_{u\in {\op{U}}_j^{\mathrm{I}}\setminus \{0\}} 
\Abs{\ln\frac{\innprd{(\widehat{\op{H}}^{\mathrm{I}}_j)^{-1} u}{u}}{\innprd{({\op{H}}^{\mathrm{I}}_j)^{-1} u}{u}}}\\
&\le & d_{\widehat{\op{U}}_j^{\mathrm{I}}}\left(\widehat{\op{M}}_j, (\widehat{\op{H}}^{\mathrm{I}}_j)^{-1}\right) +
\sup_{u\in {\op{U}}_j^{\mathrm{I}}\setminus \{0\}}
\Abs{\ln\frac{\innprd{(\widehat{\op{H}}^{\mathrm{I}}_j)^{-1} u}{u}}{\innprd{({\op{H}}^{\mathrm{I}}_j)^{-1} u}{u}}}\\
&\stackrel{\eqref{eq:specdistinv}}{=} 
& d_{\widehat{\op{U}}_j^{\mathrm{I}}}\left(\widehat{\op{S}}_j, \widehat{\op{H}}^{\mathrm{I}}_j\right) +
\sup_{u\in {\op{U}}_j^{\mathrm{I}}\setminus \{0\}}
\Abs{\ln\frac{\innprd{{\pi}^{\mathrm{I}}_j(\widehat{\op{H}}^{\mathrm{I}}_j)^{-1} u}{u}}{\innprd{({\op{H}}^{\mathrm{I}}_j)^{-1} u}{u}}}.
\eeqs
By Proposition~\ref{lma:HSapprox}, the first term above is bounded by $C_3 \beta^{-1}h_j$, assuming $\beta^{-1}h_j$ is sufficiently small.
The second term expresses the spectral distance between ${\pi}^{\mathrm{I}}_j(\widehat{\op{H}}^{\mathrm{I}}_j)^{-1}\eta_j^{\mathrm{I}}$
and $({\op{H}}^{\mathrm{I}}_j)^{-1}$, and is bounded by $C_4 \beta^{-1}\mu(\partial_{\mathrm{n}}\Omega_j^{\mathrm{I}})$ provided
$\beta^{-1}\mu(\partial_{\mathrm{n}}\Omega_j^{\mathrm{I}})$ is sufficiently small, cf. Lemma~\ref{lma:hinv-htildeinv-comp}, which concludes the proof.
\end{proof}

%% file: multigrid.tex
\section{The multigrid preconditioner}
\label{sec:multigrid}
The extension of the two-grid preconditioner introduced in Section~\ref{sec:two-grid} to a multigrid preconditioner
follows closely~\cite{doi:10.1080/10556788.2013.854356}. 
However, since the use of non-conforming spaces requires a few changes both in the construction  and the analysis,
 we give here a full description of the extension process.
As in~\cite{doi:10.1080/10556788.2013.854356}, we adopt the following point of view: the level for which we construct a multigrid 
preconditioner is given to be  $j$ and is considered fixed, and we also fix an inactive set $\op{I}^{(j)}$, which corresponds 
to one of the SSNM iterations. This leads 
to the definition of $\op{H}^{\mathrm{I}}_{j}$ as in~\eqref{eq:inacthessdef}.
As with other multigrid methods for integral equations of the second kind, the base level, denoted by $j_0$, 
may not necessarily be the coarsest case available, 
i.e., $j_0=0$, but has to sufficiently fine for the conditions in Theorem~\ref{theo:multigrid} below to be satisfied. 
The goal is to construct  the operator $\op{Z}_{j}$ representing the {\bf multigrid} preconditioner for $\op{H}^{\mathrm{I}}_{j}$, i.e., 
an approximation of $(\op{H}^{\mathrm{I}}_{j})^{-1}$.

\subsection{Construction and complexity}
The first step in building the multigrid preconditioner is to construct the coarse inactive spaces and operators for the levels
$k=j-1,  j-2, \dots, j_0$, in accordance with~\eqref{eq:coarseixdef}. More precisely, after defining
$$\Omega_{j}^{\mathrm{I}} = \bigcup_{i\in \op{I}^{(j)}} \mathrm{supp} (\varphi^{(j)}_i),$$
we construct recursively the coarser inactive index-sets, domains, and spaces.
\begin{algorithm}[Inactive set, inactive domain definition]
\label{alg:inactive_sets} 
\begin{enumerate}
\item[$1.$] {\tt for}\hspace{3pt} $k=(j-1)\::\:-1\::\:j_0$ \hspace{3pt} 
\vspace{6pt}
\item[$2.$] \hspace{20pt} $\op{I}^{(k)} \eqdef \left\{i\in\{1,\dots, N_{k}\}\  :\ \mu(\mathrm{supp}(\varphi_i^{(k)})\cap \Omega_{k+1}^{\mathrm{I}})>0\right\}$
\vspace{6pt}
\item[$3.$] \hspace{20pt} $\Omega_k^{\mathrm{I}} \eqdef \bigcup_{i\in \op{I}^{(k)}} \mathrm{supp} (\varphi^{(k)}_i)$
\vspace{6pt}
\item[$4.$] {\tt end}
\end{enumerate}
\end{algorithm}

With inactive index-sets constructed, we now define, as before, the inactive spaces and operators for $k=j_0, \dots, j$:
\beqs
\op{U}_k^{\mathrm{I}}& \eqdef& \mathrm{span}\{\varphi^{(k)}_i\ :\ i\in \op{I}^{(k)}\}\ ,\\
\op{H}^{\mathrm{I}}_k& \eqdef& \pi^{\mathrm{I}}_k (\op{K}_k^*\op{K}_k+\beta I)\op{E}_k^{\mathrm{I}} \in \mathfrak{L}(\op{U}_k^{\mathrm{I}})\ .
\eeqs
Recall that $\op{U}_k^{\mathrm{I}}\subseteq \op{U}_k$, but we do not expect in general that $\op{U}_k^{\mathrm{I}}\subseteq \op{U}_{k+1}^{\mathrm{I}}$. However,
the inclusion $\Omega_{k+1}^{\mathrm{I}}\subseteq \Omega_{k}^{\mathrm{I}}$ holds for $k=j_0, \dots, j-1$.
We also define for $k=1, 2, \dots$ the operators 
\beq
\label{eq:transopdef}
\mathfrak{I}_{k-1}^k:\mathfrak{L}(\op{U}_{k-1}^{\mathrm{I}}) \to \mathfrak{L}(\op{U}_{k}^{\mathrm{I}}),\ \ 
\mathfrak{I}_{k-1}^k(\op{X}) = \pi_{k}^{\mathrm{I}}\left(\op{X}\cdot \pi_{k-1}^{\mathrm{I}} + \beta^{-1}(I-\pi_{k-1}^{\mathrm{I}})\right).
\eeq
Note that the two-grid preconditioner $\op{M}_j$ can be written as
\beq
\label{eq:sdefop}
\op{M}_j = \mathfrak{I}_{j-1}^j\left((\op{H}^{\mathrm{I}}_{j-1})^{-1}\right)\ .
\eeq
Another essential element in defining the multigrid preconditioner is the 
family of operators $\mathfrak{N}_k$, $k=j_0, \dots, j$,  given by 
$$\mathfrak{N}_k:\mathfrak{L}(\op{U}_{k}^{\mathrm{I}})\to \mathfrak{L}(\op{U}_{k}^{\mathrm{I}}),
\ \ \mathfrak{N}_k(\op{X})\stackrel{\mathrm{def}}{=} 2\op{X}-\op{X}\cdot\op{H}_k^{\mathrm{I}}\cdot\op{X}\ .$$
It is known that $\op{X}_{l+1}=\mathfrak{N}_k(\op{X}_l)$, $l=1,2,\dots$,
represents the Newton iteration for solving the nonlinear operator-equation $\op{X}^{-1}-\op{H}_k^{\mathrm{I}} = 0$ (e.g., see~\cite{MR2429872}).

The following algorithm produces for \mbox{$k=j_0+1,\dots,j$} a sequence of operators $\op{Z}_{k}\in \mathfrak{L}(\op{U}_{k}^{\mathrm{I}})$,
of which $\op{Z}_{j}$ is the desired multigrid preconditioner. 
\begin{algorithm}[Operator-form definition of $\op{Z}_{k}$; input arguments: $j \ge j_0+1$]
\label{alg:multigrid}
\begin{enumerate}
\item[$1.$] $\op{Z}_{j_0} \eqdef (\op{H}^{\mathrm{I}}_{j_0})^{-1}$ \hspace{73pt} {\tt \%} base level
\vspace{6pt}
\item[$2.$] {\tt for} \ $k=j_0+1:j-1$ \hspace{42pt} {\tt \%} intermediate levels (if any)
\vspace{6pt}
\item[$3.$] \hspace{22pt} $\op{Z}_{k} \eqdef\mathfrak{N}_k(\mathfrak{I}_{k-1}^k (\op{Z}_{k-1}))$ \hspace{29pt}
\vspace{6pt}
\item[$4.$] {\tt end}
\vspace{6pt}
\item[$5.$] $\op{Z}_{j} \eqdef \mathfrak{I}_{j-1}^j (\op{Z}_{j-1})$ \hspace{64pt} {\tt \%} finest level
\end{enumerate}
\end{algorithm}
\vspace{5pt}

Algorithm~\ref{alg:multigrid} shows that~$\op{Z}_j$ has a W-cycle structure. Moreover, for $k<j$, applying $\op{Z}_k$
involves one application of~$\op{H}_k^{\mathrm{I}}$. 
To estimate the cost of applying~$\op{Z}_j$ we make some assumptions with respect to the
cost of applying~$\op{H}_k^{\mathrm{I}}$  and the cost of inverting $\op{H}^{\mathrm{I}}_{j_0}$ at step 1 using 
\emph{unpreconditioned conjugate gradient} (CG).
Recall that $N_k=\mathrm{dim}(\op{U}_k)$, and assume that there exists $\alpha\in (0,1)$ so that $N_{k-1}\le \alpha N_k$, $k=1, 2, \dots$; 
we expect $\alpha \approx 2^{-n}$, where $n$ is the dimension of the ambient space. We also assume that the cost of applying
the Hessian $\op{H}_k$, and hence $\op{H}^{\mathrm{I}}_k$, is 
$$t(k)\approx C_{op} N_k^p,\ \  p\ge 1.$$
 For the elliptic-constrained problem~\eqref{eq:contprobPDE} we take $p=1$ if we use classical multigrid for solving 
the elliptic problems, while for the image deblurring example we have $p=2$. We assume that the cost of applying
$\op{H}_k$ dominates the added $O(N_k)$-costs of projecting vectors onto the coarse space  and other usual vector
additions in the preconditioner, hence we discard the latter from the cost computation. 
The last hypothesis is that for any level $k$, CG converges
to the desired tolerance in at most $F_{cg}$ iterations at a cost of $c F_{cg} N_k$ flops. In practice
we have seen $F_{cg}$ to range between $10-100$ on a variety of problems.
It follows from Algorithm~\ref{alg:multigrid}
that the cost $f(k)$ of applying $\op{Z}_k$ satisfies the recursion:
\beq
\label{eq:Fjrecursion}
f(j)&\le & f({j-1}) + O(N_j) \approx f({j-1}) \\
\nonumber
f(k)&\le & 2 f({k-1}) + t(k),\ \ k=j_0+1,\dots, j-1\\
\nonumber
f({j_0})& \le & c F_{cg} N_{j_0} .
\eeq
Assuming that $2\alpha^p<1$, a standard argument shows that
\beqs
f({j-1})\le 2^{j-j_0-1} c F_{cg} N_{j_0} + C_{op}\frac{1-(2\alpha^p)^{j-j_0-1}}{1-2 \alpha^p} N_{j-1}^p.
\eeqs
If we denote by $l=j-j_0+1$ the number of levels used (i.e., $j=j_0+2$ meaning three levels)
and discard the $O(N_j)$ term in~\eqref{eq:Fjrecursion}, then
\beq
\label{eq:costmmg}
f({j})\le \left((2\alpha)^{l-1} F_{cg} \frac{c }{2C_{op}}  + \frac{\alpha^p}{1-2 \alpha^p}(1-(2\alpha^p)^{l-2}) \right) \overbrace{C_{op}N_j^p}^{t(j)}.
\eeq
The expression above is not expected to be consistent with the cases $l=1, 2$ due to the neglection of the
costs of projections.
Formula~\eqref{eq:costmmg} shows 
that it is certainly advantageus to use as many levels as possible to keep the 
cost $f(j)$ of applying the preconditioner $\op{Z}_j$ low relative to the cost $t(j)$ 
of applying the inactive Hessian $\op{H}^{\mathrm{I}}_j$. Asymptotically, 
if $l$ is large, then 
\beq
\label{eq:asymptcost}
f(j)\approx \frac{\alpha^p}{(1-2\alpha^p)}t(j)\ .
\eeq
If $\alpha$ is truly small due to high-dimensionality and/or the cost of applying the Hessian is high
(either $C_{op}\gg c$ or $p>1$), then the relative cost $f(j)/t(j)$ can be small even with a low number of levels.
We expect the wall-clock timings we show in Section~\ref{sec:numerics} to give a better picture of the computational 
savings of using the \emph{multigrid preconditioned conjugate gradient} (MGCG) versus CG. 

However, we must point out that our computations are only two-dimensional, so $\alpha \approx 1/2$. Thus, 
in order to notice significant savings in computing time, we need either high-resolution and/or many levels.
For higher dimensions (three and four), the factor $\alpha^p$ in~\eqref{eq:asymptcost} is expected to 
be significantly smaller, resulting in a much lower cost of applying the multigrid preconditioner. Thus we anticipate
that the wall-clock savings in higher dimensional problems will occur at lower resolutions as for two-dimensional problems.

\subsection{Analysis}
Estimating the spectral distance between the multigrid preconditioner $\op{Z}_{j}$ and $(\op{H}^{\mathrm{I}}_{j})^{-1}$
follows the same path as the analysis in~\cite{doi:10.1080/10556788.2013.854356}. The only significant difference lies in the
presence of the projection $\pi_{k}^{\mathrm{I}}$ in the operator $\mathfrak{I}_{k-1}^k$ defined in~\eqref{eq:transopdef}\footnote{Erratum: 
On p. 800 of~\cite{doi:10.1080/10556788.2013.854356} the correct definition is 
$\mathfrak{I}_{j-1}^j(\op{X}) = \op{X}\cdot \pi_{j-1}^{\mathrm{in}} + \beta^{-1}(I-\pi_{j-1}^{\mathrm{in}})$.}. We now verify that $\mathfrak{I}_{k-1}^k$
is non-expansive in the spectral distance, a result similar to Lemma~4.2 in~~\cite{doi:10.1080/10556788.2013.854356}.
\begin{lemma}
\label{lma:nonexpans_transop}
For $k=1, 2, \dots$, and $\op{X}\in \mathfrak{L}_+(\op{U}_{k-1}^{\mathrm{I}})$, we have 
$\mathfrak{I}_{k-1}^k(\op{X})\in \mathfrak{L}_+(\op{U}_{k}^{\mathrm{I}})$. Moreover, if $\op{X},~\op{Y}\in \mathfrak{L}_+(\op{U}_{k-1}^{\mathrm{I}})$, then 
\beq
\label{eq:nonexpans_transop}
d_{\op{U}_{k}^{\mathrm{I}}}(\mathfrak{I}_{k-1}^k(\op{X}),\mathfrak{I}_{k-1}^k(\op{Y})) \le 
d_{\op{U}_{k-1}^{\mathrm{I}}}(\op{X},\op{Y}) \ .
\eeq
\end{lemma}
\indent\emph{Proof.}
If $\op{X}\in \mathfrak{L}_+(\op{U}_{k-1}^{\mathrm{I}})$, then for $u, v\in \op{U}_{k}^{\mathrm{I}}$ we have
\beqs
\innprd{\pi_{k}^{\mathrm{I}}\op{X} \pi_{k-1}^{\mathrm{I}}u }{v} &=& \innprd{\op{X} \pi_{k-1}^{\mathrm{I}}u }{v}
= \innprd{\op{X} \pi_{k-1}^{\mathrm{I}}u }{\pi_{k-1}^{\mathrm{I}}v} = \innprd{\pi_{k-1}^{\mathrm{I}}u }{\op{X} \pi_{k-1}^{\mathrm{I}}v}\\
&=& \innprd{u }{\op{X} \pi_{k-1}^{\mathrm{I}}v} = \innprd{u }{\pi_{k}^{\mathrm{I}} \op{X} \pi_{k-1}^{\mathrm{I}}v},
\eeqs
which shows that $\pi_{k}^{\mathrm{I}} \op{X} \pi_{k-1}^{\mathrm{I}}$ is symmetric (recall that neither of $\op{U}_{k-1}^{\mathrm{I}}$
and $\op{U}_{k}^{\mathrm{I}}$ are assumed to be subspaces of each other). Given the symmetry of the orthogonal 
projection $\pi_{k}^{\mathrm{I}}(I-\pi_{k-1}^{\mathrm{I}})$ onto $(\op{U}_{k-1}^{\mathrm{I}})^\perp \cap \op{U}_{k}^{\mathrm{I}}$, 
the symmetry of $\mathfrak{I}_{k-1}^k(\op{X})$ follows. We leave the positive definiteness of $\mathfrak{I}_{k-1}^k(\op{X})$ as 
an exercise to the reader.

Let $\op{X},\ \op{Y} \in \mathfrak{L}_+(\op{U}_{k-1}^{\mathrm{I}})$. By Lemma~4.1 in~\cite{doi:10.1080/10556788.2013.854356} we have
\beq
\label{eq:logineq}
\Abs{\ln\frac{w_1+x}{w_2+x}}\le \Abs{\ln\frac{w_1}{w_2}}, \forall x>0,
\eeq 
for any $w_1, w_2$ complex numbers in the right half-plane. So
\beqs
&&d_{\op{U}_{k}^{\mathrm{I}}}(\mathfrak{I}_{k-1}^k(\op{X}),\mathfrak{I}_{k-1}^k(\op{Y})) =
\sup_{u\in \op{U}_{k}^{\mathrm{I}}\setminus\{0\}} 
\Abs{\ln\frac{\innprd{\pi_{k}^{\mathrm{I}}(\op{X}\pi_{k-1}^{\mathrm{I}}+\beta^{-1}(I-\pi_{k-1}^{\mathrm{I}}))u}{u}}
  {\innprd{\pi_{k}^{\mathrm{I}}(\op{Y}\pi_{k-1}^{\mathrm{I}}+\beta^{-1}(I-\pi_{k-1}^{\mathrm{I}}))u}{u}}}\\
&&= \sup_{u\in \op{U}_{k}^{\mathrm{I}}\setminus\{0\}} \Abs{\ln\frac{\innprd{(\op{X}\pi_{k-1}^{\mathrm{I}}+\beta^{-1}(I-\pi_{k-1}^{\mathrm{I}}))u}{u}}
  {\innprd{(\op{Y}\pi_{k-1}^{\mathrm{I}}+\beta^{-1}(I-\pi_{k-1}^{\mathrm{I}}))u}{u}}}
\stackrel{\eqref{eq:logineq}}{\le}\sup_{u\in \op{U}_{k}^{\mathrm{I}}\setminus\{0\}} \Abs{\ln\frac{\innprd{\op{X}\pi_{k-1}^{\mathrm{I}}u}{u}}
  {\innprd{\op{Y}\pi_{k-1}^{\mathrm{I}}u}{u}}} \\
= &&\sup_{u\in \op{U}_{k}^{\mathrm{I}}\setminus\{0\}} \Abs{\ln\frac{\innprd{\op{X}\pi_{k-1}^{\mathrm{I}}u}{\pi_{k-1}^{\mathrm{I}}u}}
  {\innprd{\op{Y}\pi_{k-1}^{\mathrm{I}}u}{\pi_{k-1}^{\mathrm{I}}u}}}
\le \sup_{v\in \op{U}_{k-1}^{\mathrm{I}}\setminus\{0\}} \Abs{\ln\frac{\innprd{\op{X}v}{v}}{\innprd{\op{Y}v}{v}}} = d_{\op{U}_{k-1}^{\mathrm{I}}}(\op{X},\op{Y}).\qquad \endproof
\eeqs
We also recall two technical results from~\cite{MR2429872}.
The next result appears as Lemma~5.3 in~\cite{MR2429872}.
\begin{lemma}
\label{lma:lma53}
Let $(e_k)_{k\ge 0}$ and $(a_k)_{k \ge 0}$ be positive numbers satisfying
\begin{equation}
\label{eq:mg_err_abst}
e_{k} \le C(e_{k-1} + a_{k})^2\ ,\ \  a_{k}  \le  a_{k-1} \le f^{-1} a_{k},\ \ k=1, 2, \dots, 
\end{equation}
for some $0<f<1$. If $a_0 \le \frac{f}{4 C}$ and if $e_0 \le 4 C a_0^2$, then 
\begin{equation}
e_k \le 4 C a_k^2,\ \ \forall k> 0\ .
\end{equation}  
\end{lemma}
Second, from Theorem~3.12 in~\cite{MR2429872} we extract the following result signifying the quadratic
convergence of Newton's method for the operator equation $X^{-1}-A=0$ measured in the spectral distance.
\begin{lemma}
\label{lma:lma312}
Given a Hilbert space~$\op{X}$ and $A,H\in\mathfrak{L}_+(\op{X})$ so that
\mbox{$d_{\op{X}}(A,H^{-1})< 0.4$}, we have
\beq
\label{eq:quadspecdist}
d_{\op{X}}(2 A-A H A, H^{-1}) \le 2 \left(d_{\op{X}}(A,H^{-1})\right)^2\ .
\eeq
\end{lemma}
We are now in the position to prove the main result of this section. 
\begin{theorem}
\label{theo:multigrid}
Assume that the operators $\op{K}$, $(\op{K}_j)_{j\ge 0}$ satisfy Condition$~\ref{cond:condsmooth}$, and let $0\le j_0 < j$ be fixed indices.
Consider the inactive index-sets and inactive domains defined by Algorithm~$\ref{alg:inactive_sets}$, and the sequence of
operators  $\op{Z}_{k}$, $j_0\le k \le j$ defined by
Algorithm~$\ref{alg:multigrid}$. Denote by $\mu_k=\mu(\Omega_{k-1}^{\mathrm{I}}\setminus \Omega_{k}^{\mathrm{I}})$, 
and assume there exists $0<f\le f_{low}$ so that $\mu_k\le \mu_{k-1}\le f^{-1}\mu_k$ for $k=j_0+1, j_0+2, \dots, j$, 
with $f_{low}$ given in~\eqref{eq:refinement_assump}.
If 
\beq
\label{eq:coarsestgridcond}
C_{\mathrm{tg}}\beta^{-1}(h_{j_0}+\mu_{j_0})<\min(0.1, f/8),
\eeq
then there exists
$C_{\mathrm{mg}}>0$ independent of $j$ and 
the inactive set $\widehat{\op{I}}^{(j)}$ so that
\beq
\label{eq:optordmgprec}
d_{\op{U}_{j}^{\mathrm{I}}}(\op{Z}_{j},(\op{H}^{\mathrm{I}}_{j})^{-1})\le 
C_{\mathrm{mg}} \beta^{-1}\left(h_{j} + \mu_{j}\right)\ . 
\eeq
\end{theorem}
\begin{proof}
For $j_0\le k\le j$ denote $e_k=d_{\op{U}_k^{\mathrm{I}}}(\op{Z}_{k},(\op{H}^{\mathrm{I}}_{k})^{-1})$, and 
$a_k = C_{\mathrm{tg}}\beta^{-1}\left(h_k + \mu_k\right)$.
The assumptions on $\mu_k$ and $f$ imply that
$$a_k\le a_{k-1}\le f^{-1} a_k,\ \ \forall \ \ j_0+1 \le k \le j,$$ 
and that $a_k\le 0.1$ for $j_0 \le k \le j$.
Since for $k<j$ the operator $\op{Z}_{k}$ is defined as $\mathfrak{N}_k(\mathfrak{I}_{k-1}^k (\op{Z}_{k-1}))$,
our first goal is to ensure that~\eqref{eq:quadspecdist} holds for all $k\ge j_0$ with $A=\mathfrak{I}_{k-1}^k (\op{Z}_{k-1})$
and $H= \op{H}^{\mathrm{I}}_{k}$. Thus we prove by induction that $e_k<0.2$ for $j_0\le k\le j-1$, 
and that the sequences $e_k$ and $a_k$ satisfy~\eqref{eq:mg_err_abst}
with $C=2$ for $j_0\le k< j$. Note that $e_{j_0}=0$. For $k\ge j_0+1$, after recalling that 
$\op{M}_{k} = \mathfrak{I}_{k-1}^k\left((\op{H}^{\mathrm{I}}_{k-1})^{-1}\right)$, we have 
\beq
d_{\op{U}_k^{\mathrm{I}}}(\mathfrak{I}_{k-1}^k (\op{Z}_{k-1}),(\op{H}^{\mathrm{I}}_{k})^{-1}) &\le &
d_{\op{U}_k^{\mathrm{I}}}(\mathfrak{I}_{k-1}^k (\op{Z}_{k-1}),\op{M}_{k}) +
d_{\op{U}_k^{\mathrm{I}}}(\op{M}_{k},(\op{H}^{\mathrm{I}}_{k})^{-1})\nonumber\\
\label{eq:firstboundmg}
&\stackrel{\eqref{eq:optordprec},\eqref{eq:nonexpans_transop}}{\le}& e_{k-1}+a_k \stackrel{\mathrm{inductive\  hyp.}}{\le}\ 0.2+0.1=0.3.
\eeq
So Lemma~\ref{lma:lma312} together with~\eqref{eq:firstboundmg}  implies that
\beq
\label{eq:eq:secboundmg}
e_k = d_{\op{U}_k^{\mathrm{I}}}(\mathfrak{N}_k(\mathfrak{I}_{k-1}^k (\op{Z}_{k-1})),(\op{H}^{\mathrm{I}}_{k})^{-1})<2(e_{k-1}+a_k)^2 < 2(0.3)^2<0.2\ ,
\eeq
and the inductive statement is proved.
Since $a_{j_0}< f/8$ by assumption, Lemma~\ref{lma:lma53} now implies that
$$e_{j-1}\le 8 a_{j-1}^2\ .$$
Since 
$\op{Z}_{j}=\mathfrak{I}_{j-1}^{j} (\op{Z}_{j-1})$,
it follows, as above, that
\beqs
e_{j} &\le & e_{j-1} + a_{j} \le 8 a_{j-1}^2+ a_{j} \le 
 (0.8\cdot f^{-1}+1) a_{j}\ .
\eeqs
Therefore~\eqref{eq:optordmgprec} holds with $C_{\mathrm{mg}} = (0.8\cdot f^{-1}+1) C_{\mathrm{tg}}$.
\end{proof}
\begin{remark}
\label{rem:mggridhypopt}
We should note that the hypotheses of Theorem~$\ref{theo:multigrid}$ are consistent with the
scenario discussed in Remark~$\ref{rem:twogridhypopt}$ under which the correct inactive domain $\Omega^{\mathrm{I}}$
is sufficiently regular and the sets $\Omega^{\mathrm{I}}_k$, $j_0\le k\le j$, approximate $\Omega^{\mathrm{I}}$ sufficiently 
well so that $\mu_k \approx C h_k$. Under these conditions, Theorem~$\ref{theo:multigrid}$ also shows that 
the multigrid preconditioner is of optimal order, assuming that the coarsest grid $j_0$ is sufficiently fine 
for~\eqref{eq:coarsestgridcond} to hold.
\end{remark}

%% file: numerics.tex
\section{Numerical experiments}
\label{sec:numerics}
We test our multigrid preconditioner on two problems. In Section~\ref{ssec:numerics_ell} we consider
a classical elliptic-constrained optimization problem, while in  Section~\ref{ssec:numerics_deblur} 
we showcase the behavior of our algorithm on a constrained optimization method related to image deblurring.
Essentially, in these numerical experiments we are looking, first, for a validation of our theoretical results 
and, second, we would like to estimate the practical value of our preconditioning technique. With respect to the first aim 
we would like to see that the two-grid preconditioner gives rise to a number of linear iterations per SSNM 
step that decreases (in average) with respect to increasing resolution. A similar behaviour is expected to hold
for three-grid preconditioners, four-grid preconditioners, etc; we call this the \emph{weak test}, and we expect
all computations to pass this.
We are also interested to see if the experiments pass the following
\emph{strong test}: for a fixed, acceptable (cf. Theorem~\ref{theo:multigrid})
base level $j_0$, we should observe the number of linear iterations per SSNM 
to be decreasing with an increasing number of levels. The strong test is expected to hold only asymptotically in general,
since $C_{\mathrm{mg}}$ from Theorem~\ref{theo:multigrid} is  larger than $C_{\mathrm{tg}}$ from 
Theorem~\ref{th:optordprec}; this normally results in an increase in number of iterations
from two-grid to three-grid preconditioning, only to begin decreasing when the number
of levels is sufficiently large. If the multigrid preconditioner passes the strong test for a given set of parameters, 
then we expect to see an increase in wall-clock efficiency as well.
We also expect that the multigrid preconditioner is inefficient or even fails
if the base level resolution $h_{j_0}$ is too large relative to $\beta$. With respect to the second aim we simply 
want to observe the wall-clock efficiency of the multigrid preconditioner. All computations were performed using M{\footnotesize ATLAB} 
on a system with two eight-core 2.9 GHz Intel  Xeon E5-2690 CPUs and 256 GB memory.

\subsection{An elliptic-constrained optimal control problem}
\label{ssec:numerics_ell}
For the first numerical experiment we consider the classical elliptic-constrained optimization
problem
\begin{eqnarray}
\label{eq:ellptic_constrained}
\left\{
  \begin{array}{l}\vspace{7pt}
    \min_{u\in L^2(\Omega)} \frac{1}{2}\norm{y-y_d}^2 +
    \frac{\beta}{2}\norm{u}^2\\   \textnormal{subject to\ \ } -\Delta
    y=u\ \ \mathrm{(weakly)},\ \ y\in H_0^1(\Omega),\  0\le u \le 1\
    \  a.e.\  \mathrm{in}\ \Omega,
  \end{array}\right .
\end{eqnarray}
where $\Delta$ is the Laplace operator acting on $H_0^1(\Omega)$ with $\Omega=(0,1)\times(0,1)\subset \R^2$.
Therefore, $\op{K}=(-\Delta)^{-1}$.
We define the data by $y_d=\op{K} u_d$, where the so-called target control $u_d$ is the step function shown in the left-side of
Figure~\ref{fig:uduminb4}. Note that $u_d$ is bounded between $0$ and $1$, and is supported inside the domain $\Omega$. Naturally, 
for $\beta\ll 1$ we expect $u_{\min}\approx u_d$. In absence of any box-constraints, or when the constraints turn out to be
everywhere inactive,  the solution $u_{\min}$ of~\eqref{eq:ellptic_constrained} also solves the Tikhonov-regularized 
inverse problem $\op{K}u=y_d$.
It is well known that in this case $u_{\min}$ may  not be localized and can  exhibit an oscillatory behavior near the support of $u_d$. 
In order to showcase the behavior of our algorithm we selected a range of values for~$\beta$ that render
the constraints to be active on a significant portion of $\Omega$ (which requires a sufficiently small~$\beta$),  
while allowing at the same time  for a relatively fast convergence, e.g., less than ten SSNM iterations.
We thus present results for $\beta=10^{-4}, 10^{-5}$, and~$10^{-6}$ in 
Tables~\ref{tab:new_method_b4},~\ref{tab:new_method_b5}, and~\ref{tab:new_method_b6}, respectively.
For the $\beta$-values listed we show the solution $u_{\min}$ in Figures~\ref{fig:uduminb4} (right image) 
and~\ref{fig:uduminb5}. For $\beta=10^{-6}$ (Figure~\ref{fig:uduminb5}, right) both constraints are 
active at the solution, while for $\beta=10^{-4}$ and $\beta=10^{-5}$ only the lower constraints are active.

\begin{figure}[!htb]
        \includegraphics[width=5.3in]{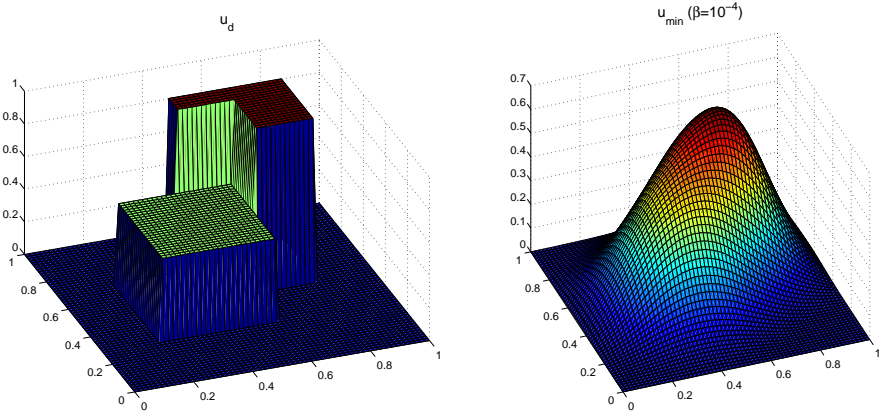}
\label{fig:uduminb4}
\caption{Left: target control $u_d$. Right: optimal control $u_{\min}$ for $\beta=10^{-4}$.}
\end{figure}

\begin{figure}[!htb]
\begin{center}
        \includegraphics[width=5.3in]{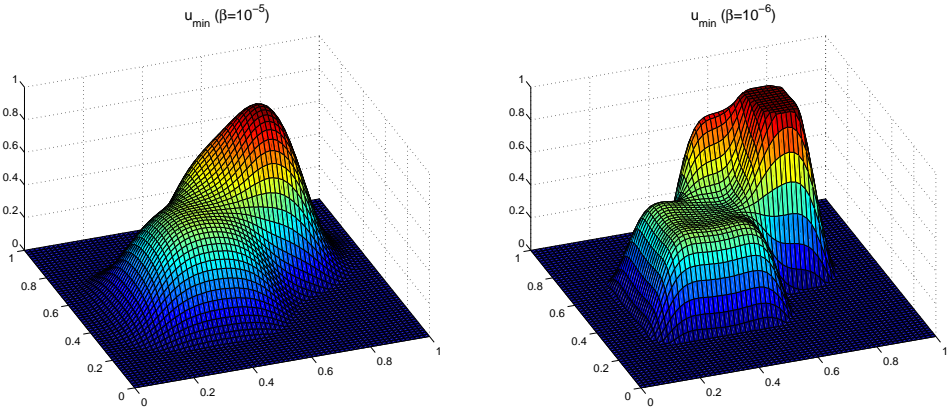}
\caption{Left: optimal control $u_{\min}$ for $\beta=10^{-5}$. Right: optimal control $u_{\min}$ for $\beta=10^{-6}$.}
\label{fig:uduminb5}
\end{center}
\end{figure}

Given $n\in\N$, we divide $\Omega$ uniformly in $n^2$ squares and we discretize the control space
using piecewise constant functions; the departure from the theoretical framework in the earlier sections is minimal, we just
replaced triangular elements with rectangular ones.
We then use a standard Galerkin formulation to produce a discrete version of $\op{K}$ on each grid using continuous bilinear
finite elements. Standard finite element analysis (e.g., see~\cite{MR2373954}) shows that the SAC 
Condition~\ref{cond:condsmooth} is satisfied; in particular, part [c] of Condition~\ref{cond:condsmooth} follows from the
$H^2$-regularity of the elliptic equation coupled with  $L^{\infty}$-convergence (see also~\cite{MR2322235}).
For each 
$$n_j=64\times 2^j, \ \ j=0,\dots, 6,$$
we initialize the SSNM using the solution obtained from a coarser level; we solve
the linear systems in the SSNM solution process using MGCG, and we compare the 
results against CG. For each run, we report in Tables~\ref{tab:new_method_b4},~\ref{tab:new_method_b5}, 
and~\ref{tab:new_method_b6} the average number of MGCG/CG iterations per SSNM step as well as the 
added wall-clock times used by the MGCG/CG solves during the entire solution process. The relative
tolerance for the linear solves is set at $10^{-8}$. The elliptic problem, i.e., the application of $\op{K}_j$ and $\op{K}_j^*$ needed
for applying the inactive Hessian $\op{H}_j^{\mathrm{I}}$, is solved numerically using either
direct methods (for $n\le 256$) or classical multigrid (the full approximation scheme FAS)
using a relative tolerance of $10^{-8}$; the base case for FAS was taken to be $n=256$, a choice that
effectively minimized wall-clock times for solving the elliptic problem on our system. 
For solving the base case (Step 1 in Algorithm~\ref{alg:multigrid}) in the multigrid preconditioner application
we use (unpreconditioned) CG with a matrix-free application of $\op{H}_{j_0}^{\mathrm{I}}$, and a tolerance of $10^{-10}$.
We should emphasize that the multigrid FAS for solving the elliptic problem is used only for applying $\op{K}_j$ and $\op{K}_j^*$,
and is completely independent from the multigrid preconditioner from Algorithm~\ref{alg:multigrid}, although in the implementation
they share part of the  infrastructure.

\begin{table}[!h]
 \begin{center}
 \caption{Comparison of iteration counts and runtimes for MGCG vs. CG; $\beta=10^{-4}$.}
      {\begin{tabular}{|@{}c|c|c|c|c|c|c|c|}\hline
          $n_j$ & $128$ & $256$ & $512$ &  $1024$ & $2048$ & 4096 & 8192\\\hline
          \# cg / it.    & 11.25 & 11.67 & 12 & 12 & 12 & 12 & 12 \\
          $t_{\mathrm{cg}}$ (s)  & 3.65 & 14 & 84 & 427 & 1915 & 2.97 h & 20.3 h \\\hline\hline
          \# mg / it., $j_0=0$ & 5 & 5 & 4 & 4 & 3 & 3 & 3\\
          $t_{\mathrm{mg}}$ (s) & 10.5 & 13.6 & 100 & 373 & 1242 & 1.41 h & 5.97 h \\\hline
          \emph{eff}=$t_{\mathrm{mg}}/t_{\mathrm{cg}}$& 2.87 & 1.16 & 1.19  & 0.87 & 0.65 & 0.48& 0.29 \\\hline
      \end{tabular}}
      \label{tab:new_method_b4}
\end{center}
\end{table}

\begin{table}[!h]
 \begin{center}
  \caption{Comparison of iteration counts and runtimes for MGCG vs. CG; $\beta=10^{-5}$.}
      {\begin{tabular}{|@{}c|c|c|c|c|c|c|c|}\hline
          $n_j$ & $128$ & $256$ & $512$ &  $1024$ & $2048$ & 4096 & 8192 \\\hline
          \# cg / it.    & 20 & 19.5 & 19.25 & 19.75 & 20 & 20 & 20 \\
          $t_{\mathrm{cg}}$ (s)  & 6.24 & 42 & 197 & 793 & 3081 & 4.88 h & 32.76 h\\\hline\hline
          \# mg / it., $j_0=0$ & 7.75 & 9 & 8.25 & 7.75 & 6 & 6.67 & 8 \\
          $t_{\mathrm{mg}}$ (s)  & 15.23 & 32 & 257 & 896 & 1949 &  2.28 h & 10.84 h \\\hline
          \emph{eff}=$t_{\mathrm{mg}}/t_{\mathrm{cg}}$& 2.44  & 0.76 & 1.3 & 1.12& 0.63 & 0.47 & 0.33 \\\hline\hline
          \# mg / it., $j_0=1$ & - & 6.5 & 7 & 6 & 5 & 4 & 5 \\
          $t_{\mathrm{mg}}$ (s)  & - & 63 & 254 & 796 & 1865 & 1.87 h & 8.4 h\\\hline
          \emph{eff}=$t_{\mathrm{mg}}/t_{\mathrm{cg}}$& - & 1.5 & 1.29 & 1.004 & 0.605 & 0.38 & 0.26\\\hline
      \end{tabular}}
      \label{tab:new_method_b5}
\end{center}
\end{table}

\begin{table}[!h]
  \begin{center}
 \caption{Comparison of iteration counts and runtimes for MGCG vs. CG; $\beta=10^{-6}$.}
      {\begin{tabular}{|@{}c|c|c|c|c|c|c|c|}\hline
          $n_j$ & $128$ & $256$ & $512$ &  $1024$ & $2048$ & 4096 & 8192 \\\hline
          \# cg / it.      & 32.5 & 31.75 & 31 & 31.75 & 33.25 & 33 & 34\\
          $t_{\mathrm{cg}}$ (s)  & 9.9 & 48 & 241 & 1312 & 1.92 h & 11.72 h & 54.58 h\\\hline\hline
          \# mg / it., $j_0=2$         & - & - & 9.75 & 11.75 & 16 & 11.25 & $> 50$ \\
          $t_{\mathrm{mg}}$ (s)  & - & - & 1135& 1986 & 2.06 h & 6.98 h & -\\\hline
          $t_{\mathrm{mg}}/t_{\mathrm{cg}}$& - & - & 4.71 & 1.51 & 1.07 & 0.59 & -  \\\hline\hline
          \# mg / it., $j_0=3$         & - & - & - & 8.75  & 10 & 12 & 11 \\
          $t_{\mathrm{mg}}$ (s)  & - & - & - & 4155 & 1.78 h & 5.27 h & 16.74 h \\\hline
          $t_{\mathrm{mg}}/t_{\mathrm{cg}}$& - & - & - & 3.16 & 0.92 & 0.45 & 0.31\\\hline\hline
          \# mg / it., $j_0=4$         & - & - & - & -  & 6.25 & 7.75 & 8.33\\
          $t_{\mathrm{mg}}$ (s)  & - & - & - & - & 4.61 h & 7.69 h & 17.44 h\\\hline
          $t_{\mathrm{mg}}/t_{\mathrm{cg}}$& - & - & - & - & 2.39 & 0.66 & 0.32 \\\hline
      \end{tabular}}
      \label{tab:new_method_b6}
\end{center}
\end{table}

First we remark that, for each $\beta$, the SSNM converged in a relatively mesh-independent number of iterations; that number
is also independent of the way we solve the linear systems, assuming they are solved to the given tolerance. In the interest of 
the exposition we do not report the number of SSNM iterations, since the focus is on the linear solves.
We also point out that all cases pass the weak test. This is best seen in Table~\ref{tab:new_method_b6}
for $\beta=10^{-6}$, where we note the average number of 
two-grid iterations decreasing from $9.75$ at $n=512$ to $8.75$ at $n=1024$, and down to $6.25$ at $n=2048$; we did not run
the two-grid preconditioned problem for $n=4096, 8192$. Still for $\beta=10^{-6}$ we see the three-grid  average number of 
iterations decreasing from $11.75$ at $n=1024$ to $10$ at $n=2048$, down to $7.75$ at $n=4096$, and the four-grid 
average number of iterations decreasing from $16$ at $n=2048$, to $12$ at $n=4096$, down to $8.33$ at $n=8192$.

For the strong test the  key issue is the choice of the base case $j_0$ for the multigrid preconditioner. 
The hypotheses of Theorem~\ref{theo:multigrid}
show that the base level has to be sufficiently fine (relative to $\beta$)
in order for MGCG to run efficiently, as shown in~\eqref{eq:coarsestgridcond}. 
In Table~\ref{tab:new_method_b4}, for \mbox{$\beta=10^{-4}$}, the choice $j_0=0\  (n_0=64)$ seems to be sufficiently fine, as the 
MGCG requires fewer and fewer iterations as $n$ increases, as predicted by theory. The \emph{effective efficiency factor}
\emph{eff}=time(MGCG) / time (CG)  is also presented; it is shown 
to decrease with increasing resolution, but it decreases below the value one 
(e.g., MGCG becomes more efficient than CG) only at higher resolution, 
as expected. For example, at $n=2048$, while CG required an average number of 12 iterations per SSNM iteration (actually 
exactly 12 at each iteration), the 5-grid MGCG required an average of 3 iterations per SSNM iterations. In terms 
of wall-clock time, the linear solves for MGCG required 0.65 of the wall-clock time of CG.
The situation is somewhat similar for $\beta=10^{-5}$ (Table~\ref{tab:new_method_b5}), except for the fact 
that $j_0=0$ turns out to be borderline acceptable, in that
the average number of  MGCG iterations does not decrease with increasing resolution right from the beginning, so $j_0=0$
does not pass the strong test. Instead, the case $j_0=1$ $(n_1=128)$ clearly passes the strong test with the exception
of the mild increase in number of iterations from two-grid to three-grid. Also, 
the efficiency factor decreases to $0.38$ at $n=4096$ (with  a five-grid preconditioner), and further down to
$0.26$ at $n=8192$ (with  a six-grid preconditioner).
Finally, for $\beta=10^{-6}$ (Table~\ref{tab:new_method_b6}) we see that neither of the values 
$j_0=2, 3, 4$ give rise to the expected decrease in the number of iterations
for the MGCG, at least not for small number of levels, thus failing the strong test.
However, for high-resolution computations  MGCG is still more 
efficient than CG: for example, a five-grid MGCG based solve at $n=8192$ ($j_0=3$)
requires an average of 11 inner iteration per SSNM iteration and $0.31$ of the time needed 
for the 34 inner CG iterations per SSNM iteration.

\subsection{Image deblurring with box constraints}
\label{ssec:numerics_deblur}
For the second application 
we define the restricted Gaussian blurring operator for functions $u\in L^1(\R^2)$  by
\beq
\label{eq:blurdef}
\op{K}^{\sigma, w}u(x)=\frac{\alpha_w^{-1}}{2 \pi}\int_{\abs{x-y}_{\infty}<w}G_{\sigma}(x-y)u(y)dy,
\eeq
where $\sigma, w>0$, $\abs{x}_{\infty}=\max(\abs{x_1}, \abs{x_2})$, $\abs{x}$ is the Euclidean norm,
\beqs
G_{\sigma}(x)=\sigma^{-2} e^{-\frac{|x|^2}{2\sigma^2}},\ \ \ \ \alpha_w = \frac{1}{2 \pi}\int_{\abs{x}_{\infty}<w} G_{\sigma}(x) dx\ .
\eeqs
Note that $\lim_{w\to \infty} \alpha_w=1$. 
Here $u:\R^2\to [0,1]$ is a function representing a grey-scale image. 
If $D_w$ denotes the square $(-w,w)\times(w,w)$ and $\chi_{D_w}$ is its characteristic function, then 
$$
\op{K}^{\sigma, w} u  = \alpha_w^{-1}\left(G_{\sigma}\cdot \chi_{D_w}\right) * u,
$$
where the convolution is defined using the rescaled Lebesgue measure $(2 \pi)^{-1}$ (see Appendix~\ref{sec:appendixdeblur}).
We remark that in the usual definition of Gaussian blurring, the domain of integration in~\eqref{eq:blurdef}
is the entire space $\R^2$. In practice, however, the integral is restricted as shown in~\eqref{eq:blurdef}, the
usual choice being $w=3\sigma$, and the ``image'' $u$ is restricted to a bounded domain~$\Omega$. As in the previous section, we
consider \mbox{$\Omega=(0,1)\times (0,1)$}, which mainly allows for two options for defining $\op{K}^{\sigma, w}$ on $L^1(\Omega)$: 
first we can extend $u\in L^1(\Omega)$ with zero outside $\Omega$, case in which $\op{K}^{\sigma, w}u$ is defined on the entire 
space~$\R^2$; furthermore, we restrict $\op{K}^{\sigma, w}u$ to $\Omega$. This gives rise, as shown in Appendix~\ref{sec:appendixdeblur}, to 
a bounded operator $\op{K}^{\sigma, w}\in \mathfrak{L}(L^2(\Omega))$. The second option is to {\bf not} extend $u\in L^1(\Omega)$ outside of 
$\Omega$, which natually results in an operator $\widehat{\op{K}}^{\sigma, w}\in \mathfrak{L}(L^2(\Omega), L^2(\Omega_w))$, where
$\Omega_w=(w,1-w)\times(w,1-w)$ (we require $0<w<1/2$). In our numerical experiments we used a discretization 
of $\widehat{\op{K}}^{\sigma, w}$, while, for convenience, we conduct the analysis in Appendix~\ref{sec:appendixdeblur}  for the former case, 
namely $\op{K}^{\sigma, w}$.
For the remainder of this section we discard the superscripts $\sigma, w$, i.e., $\op{K}=\op{K}^{\sigma, w}$.

In our numerical solution of the optimization problem~\eqref{eq:contprob}
we discretize $u$, as before, using piecewise constant functions on a uniform 
$n\times n$ grid on $\Omega$. With $h=1/n$ being the grid size, we compute the discrete 
version $\op{K}_hu$  (representing the blurred image) at the cell centers 
using a cubature rule to integrate~\eqref{eq:blurdef} numerically. Essentially, the value
of $(\op{K}_h u)$ at a node $z_k$ (the center of an element) 
is a weighted average of the values of $u$ in all squares that are at most $w_h$ away from $z_k$ (in the $\abs{\:\cdot\:}_{\infty}$-distance), 
with the weights being computed 
using the function $G_{\sigma}$ and rescaled to add up to 1; $w_h$ is a discrete version of $w$. The details of the
discretization, as well as the verification of Condition~\ref{cond:condsmooth}, are given in Appendix~\ref{sec:appendixdeblur}.
As customary in Gaussian filtering, the separability of the kernel $G$ allows for a more efficient
implementation, namely  $$\op{K}_h=\op{K}^{x_1}_h \op{K}^{x_2}_h = \op{K}^{x_2}_h \op{K}^{x_1}_h\ ,$$
where $\op{K}^{x_1}_hu$ (resp., $\op{K}^{x_2}_hu$) 
defines the application of a Gaussian filter to the image $u$ in the $x_1$-direction only (resp., $x_2$-direction). 

The setup and result presentation is similar to the experiments presented in Section~\ref{ssec:numerics_ell}.
We consider the case when $w=0.1$ and $\sigma=w/3$.
Again, we define the data by $y_d=\op{K} u_d$, where the target control $u_d$, i.e., the original image (shown
as a surface), is the same step function as in the previous experiment.
In Figure~\ref{fig:udyd64} we show both $u_d$ (left) and the blurred image $y_d=\op{K} u_d$ (right) as surfaces.
For the constrained optimization problem we use the constant constraints $a(x)=0$ and $b(x)=1$.
In Figures~\ref{fig:umindeblrcon64v1} and~\ref{fig:umindeblrcon64v2} we show the solutions of the constrained problem for
$\beta=0.04, 0.02, 0.01, 0.005$.

\begin{figure}[!htb]
         \includegraphics[width=5.3in]{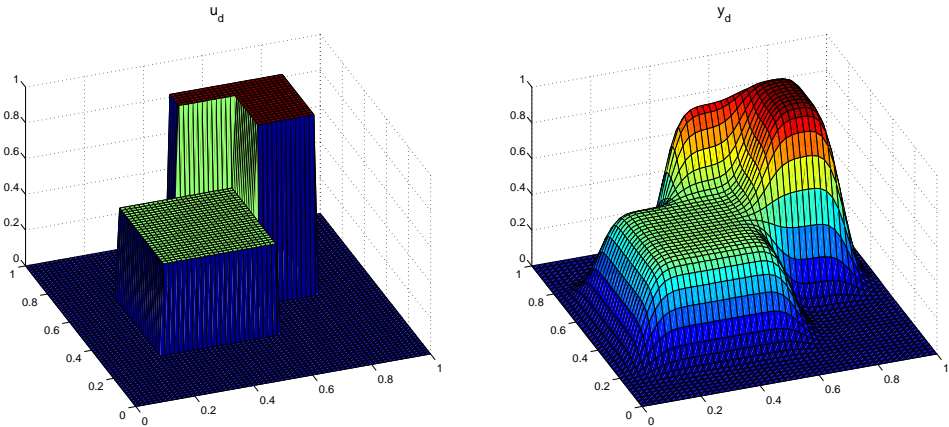}
\caption{Left: target control $u_d$. Right: blurred image/surface $y_d$.}
\label{fig:udyd64}
\end{figure}
 
\begin{figure}[!htb]
        \includegraphics[width=5.3in]{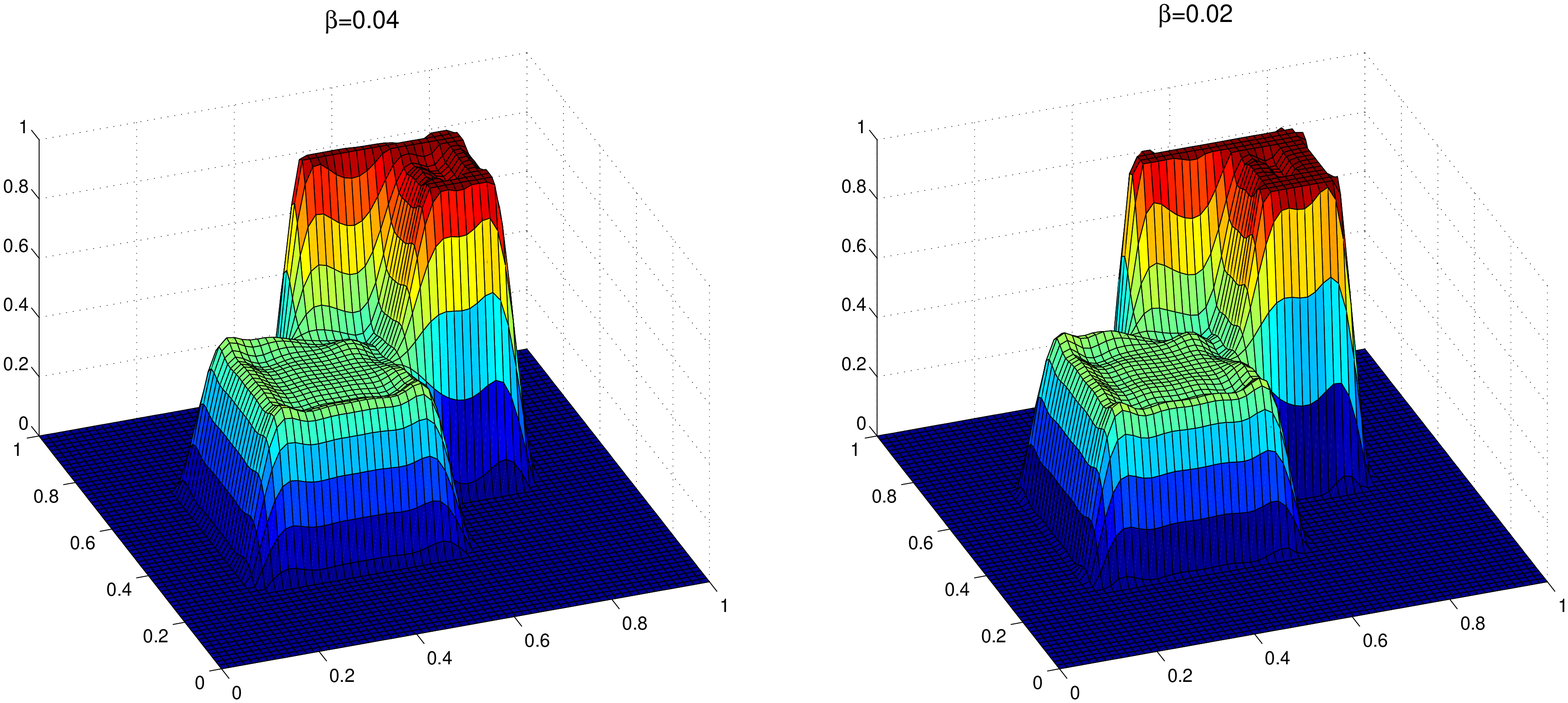}
\caption{Left: optimal control $u_{\min}$ for $\beta=0.04$. Right: optimal control $u_{\min}$ for $\beta=0.02$.}
\label{fig:umindeblrcon64v1}
\end{figure}
 
\begin{figure}[!htb]
        \includegraphics[width=5.3in]{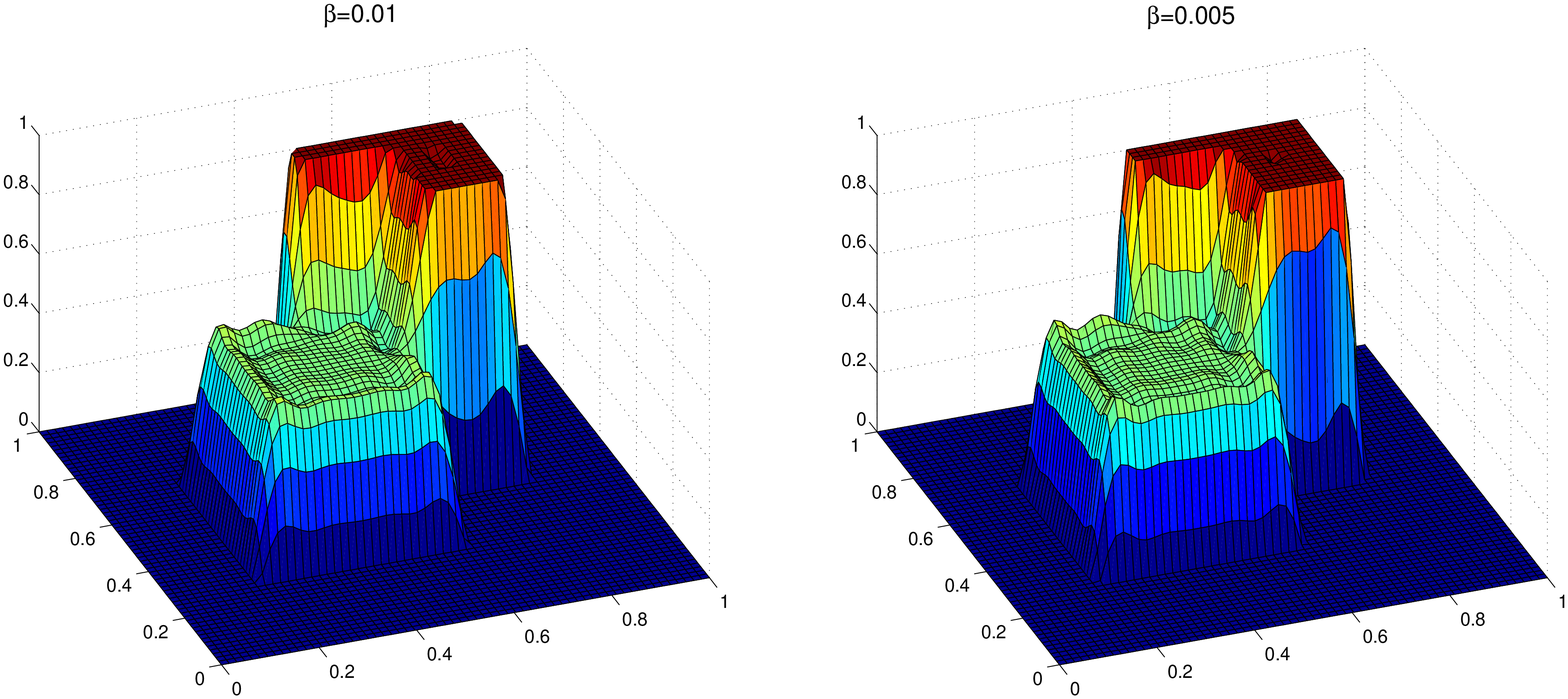}
\caption{Left: optimal control $u_{\min}$ for $\beta=0.01$. Right: optimal control $u_{\min}$ for $\beta=0.005$.}
\label{fig:umindeblrcon64v2}
\end{figure}

\begin{table}[!h]
  \begin{center}
 \caption{Comparison of iteration counts and runtimes for multigrid vs. unpreconditioned CG for image deblurring; 
   $w=0.1, \beta=0.04, 0.02$.}
      {\begin{tabular}{|@{}c|c|c|c|c|}\hline
          $n_j$ & $256$ & $512$ &  $1024$ & $2048$ \\\hline
          \multicolumn{5}{|c|}{$\beta=0.04$}\\\hline
          \# cg / it.      & 40  & 40  & 40  & 40 \\
          $t_{\mathrm{cg}}$ (s)   & 3.6  & 25  & 215  & 4023 \\\hline
          \# mg / it., $j_0=0$ & 12.2  & 14.5  & 21   & 12.5  \\
          $t_{\mathrm{mg}}$ (s)   & 7.2  & 18 &  149 & 1442    \\\hline
          \# mg / it., $j_0=1$ & -  & 9.25 & 11.5 & 10.25 \\
          $t_{\mathrm{mg}}$ (s)   &  - & 39 & 101  &  1240   \\\hline
          \# mg / it., $j_0=2$ & -  & - & 7.5   & 8.75 \\
          $t_{\mathrm{mg}}$ (s)   &  - & - & 290  &  1347   \\\hline
          \# mg / it., $j_0=3$ & -  & - & -  & 6.5 \\
          $t_{\mathrm{mg}}$ (s)   &  - & - & -  &  2665   \\\hline
          \multicolumn{5}{|c|}{$\beta=0.02$}\\\hline
          \# cg / it.   &  51.2  &  51  & 51  & 51  \\
          $t_{\mathrm{cg}}$ (s)   & 4.6  & 32  & 270  & 5354 \\\hline
          \# mg / it., $j_0=0$   & 15.8  & 18.75  & 61.75  & 57 \\
          $t_{\mathrm{mg}}$ (s)   & 12  & 23  & 450  &  6256 \\\hline
           \# mg / it., $j_0=1$   & -  &  11 & 14  & 23.75 \\
          $t_{\mathrm{mg}}$ (s)   & -  & 53  & 142  &  2797 \\\hline
            \# mg / it., $j_0=2$   & -  &  - & 9.5  & 11.75 \\
          $t_{\mathrm{mg}}$ (s)   & -  & -  & 418  &  1682 \\\hline
            \# mg / it., $j_0=3$   & -  &  - &   & 7.25  \\
          $t_{\mathrm{mg}}$ (s)   & -  & -  & -  &  3072 \\\hline
   \end{tabular}}
      \label{tab:new_method_deblur_june7}
\end{center}
\end{table}

   \begin{table}[!h]
  \begin{center}
 \caption{Comparison of iteration counts and runtimes for multigrid vs. unpreconditioned CG for image deblurring; 
   $w=0.1, \beta=0.01, 0.005$.}
      {\begin{tabular}{|@{}c|c|c|c|c|}\hline
          $n_j$ & $256$ & $512$ &  $1024$ & $2048$ \\\hline
          \multicolumn{5}{|c|}{$\beta=0.01$}\\\hline
          \# cg / it.      & 65.8   & 66  & 66  & 66  \\
          $t_{\mathrm{cg}}$ (s)   & 6  & 60  & 446  & 6576 \\\hline
          \# mg / it., $j_0=0$   &  20.8 & 25.2  & $>100$  & $>100$ \\
          $t_{\mathrm{mg}}$ (s)   & 18  &  46 &   &  \\\hline
           \# mg / it., $j_0=1$   &  - & 15.4  & 19.8   & $>100$  \\
          $t_{\mathrm{mg}}$ (s)   & -  &  119 & 279  &   \\\hline
            \# mg / it., $j_0=2$   &  - & -  & 11.8  & 15.5 \\
          $t_{\mathrm{mg}}$ (s)   & -  &  - & 846  & 2316 \\\hline
             \# mg / it., $j_0=3$   &  - & -  & -  & 9 \\
          $t_{\mathrm{mg}}$ (s)   & -  &  - & -  & 4390  \\\hline
          \multicolumn{5}{|c|}{$\beta=0.005$}\\\hline
          \# cg / it.      & 84.67   & 84.25 & 84.6 & 84.74 \\
          $t_{\mathrm{cg}}$ (s)   &  9 & 52  & 560  & 8926 \\\hline
          \# mg / it., $j_0=0$   & 31.5 & 37.5   & failed  & failed \\
          $t_{\mathrm{mg}}$ (s)   &  42 & 64  &  - &  - \\\hline
          \# mg / it., $j_0=1$   & -  & 19.75   &  26.2 & failed \\
          $t_{\mathrm{mg}}$ (s)   &  -  & 152  &  402 &  - \\\hline
          \# mg / it., $j_0=2$   & -  & -  & 15.2  & 20.25  \\
          $t_{\mathrm{mg}}$ (s)   & -  & -  & 1218  &  3260 \\\hline
          \# mg / it., $j_0=3$   & -  & -  & -  & 11.5  \\
          $t_{\mathrm{mg}}$ (s)   & -  & -  & -  &  6779 \\\hline
 \end{tabular}}
      \label{tab:new_method_deblur2_june7}
\end{center}
\end{table}

For the multigrid solves we consider the cases
$$n_j=128\times 2^j, \ \ j=0,\dots, 4.$$
We report the results for $\beta=0.04,~0.02$ in Table~\ref{tab:new_method_deblur_june7} and for
$\beta=0.01,~0.005$ in Table~\ref{tab:new_method_deblur2_june7}, but we no longer report the 
effective efficiency factor.

The results are essentially similar with the elliptic-constrained experiments. All cases clearly pass the
weak test. However, the only case where there is a hint of the strong test being passed is for $\beta=0.04$ 
(top half of Table~\ref{tab:new_method_deblur_june7}); 
for $j_0=0$ we see the average number of iterations first increasing with resolution from $12.2$ ($n_1=256$)  to $14.5$ ($n_2=512$)
up to $21$ ($n_3=1024$), only to decrease to $12.5$ for ($n_4=2048$), all compared to an average number of $40$ CG iterations. This certainly 
reflected in the wall-clock efficiency: the five-level MGCG linear solves required 1442 seconds compared to the 4023 seconds for CG.

As in the elliptic-constrained experiments, by lowering $\beta$  to $0.02$ (bottom half of Table~\ref{tab:new_method_deblur_june7}) 
we also have to raise the base case level in order
for MGCG to run efficiently; here $j_0=2$ seems to be sufficiently fine, but even $j_0=1$ seems to be acceptable, i.e., lead to  
reasonably efficient linear solves. By contrast we see how even lower values for $\beta$ (see Table~\ref{tab:new_method_deblur2_june7})
lead to very slowly convergent
linear solves (see the cases $\beta=0.01$ and $j_0=0, 1$) or even non-convergence
($\beta=0.005$ and $j_0=0, 1$).

\section{Conclusions}
\label{sec:conclusions}
We have developed a multigrid preconditioning technique to be used in connection to 
SSNMs for certain control-constrained distributed optimal control problems.
The multigrid preconditioners exhibit a provably optimal order behavior with respect to the
mesh-size, in that the quality of the preconditioners increases at the optimal rate
with increasing mesh-size, assuming a piecewise constant representation of the control and
a sufficiently fine base level. The technique used in this paper is not limited to control-constrained
problems like~\eqref{eq:contprob}. An immediate  application would be to replace (or add) a
domain-constraint to the control $u$ of the type supp($u)\subseteq \Omega'$, where
$\Omega'\subset \Omega$. Naturally, our method can be also used for PDE-constrained optimization with state
constraints by reducing them to control-constrained problem via Lavrentiev regularization.

A natural question is whether the method can be extended to 
higher order discontinuous piecewise polynomial discretizations  such that the optimality of the
preconditioner is preserved. Following the analysis of the piecewise constant case, it is apparent that
the answer is negative. However, this does not preclude the existence of alternate optimal order
preconditioners for higher order discretizations of the controls. The search for such preconditioners
is subject of ongoing research.

%% file: appendixdeblurr.tex
\section{Verification of Condition~SAC for the restricted Gaussian blurring operator}
\label{sec:appendixdeblur}
In this section we rigorously specify the  discretization for the integral operator $\op{K}^{\sigma, w}$
defined in Section~\ref{ssec:numerics_deblur},
and we show that Condition~\ref{cond:condsmooth} is satisfied. Recall that $\Omega=(0,1)\times (0,1)$
with $0<w<1/2$.

\subsection{Estimates for the continous operator}
\label{ssec:contop}
Due to the definition of $\op{K}^{\sigma, w}$ as a convolution, 
we prefer to verify Condition~\ref{cond:condsmooth}~{\bf [a]} using Fourier transforms. 
Following~\cite{MR1157815}, we consider the normalized Lebesgue measure on $\R^n$ defined by $$dm_n(x)=(2\pi)^{-n/2} dx\ ,$$
and we define the Fourier transform of a function  $f\in L^1(\R^n)$ by
$$
\op{F}_n\lbrack f \rbrack (\xi) =  \int_{\R^n} f(x) e^{-{\bf i}\: \xi\cdot x} dm_n(x)\ .
$$
In this section $L^2$-norms of functions in $\R^n$ or on bounded domains, as well as convolutions, are computed using the measure $dm_n$, i.e.,
$$
\nnorm{f}^2_{L^2(\R^n)} = \int_{\R^n} \abs{f(x)}^2 dm_n(x),\ \ \ (f * g)(x) \eqdef \int_{\R^n} f(x-y)\: g(y) dm_n(y)\ .
$$

\begin{lemma} There exists a constants $C_1, C_2, C_3$ depending only on the ratio $w/\sigma$ so that 
\label{lma:ftlinfnorm}
\beq
\label{eq:ftlinfnorm}
\Abs{\op{F}_1\lbrack \chi_{(-w,w)}(x) \cdot e^{-\frac{x^2}{2 \sigma^2}}\rbrack(\xi)} \le \min\left(C_1 / \abs{\xi},\ \sigma\: C_2\right )\ ,
\ \ \forall \xi\ne 0,
\eeq
and
\beq
\label{eq:ftlinfnorm2}
\Abs{\op{F}_1\lbrack  \chi_{(-w,w)}(x) \cdot e^{-\frac{x^2}{2 \sigma^2}}\rbrack(\xi)} \le \frac{\sigma C_3}{1+\sigma\Abs{\xi}}\ ,
\ \ \forall \xi\in \R.
\eeq
\end{lemma}
\begin{proof}
Cf.~\cite{MR1157815}, for $\xi\in \R$ the following hold:
\beqs
\op{F}_1\lbrack \chi_{(-w,w)}(x)\rbrack (\xi)&=&\sqrt{\frac{2}{\pi}}\: \frac{\sin (w\: \xi)}{\xi}\ \ \mathrm{and}\ \ 
\op{F}_1\lbrack e^{-\frac{x^2}{2 \sigma^2}}\rbrack(\xi) = \sigma \: e^{-\frac{(\sigma \xi)^2}{2}}\ ,
\eeqs
where $\sin(w\xi)/\xi$ is  continued analytically at $\xi=0$. It follows that
\beq
\nonumber
\lefteqn{ (\sigma^{-1}\: \pi)\Abs{\op{F}_1\lbrack \chi_{(-w,w)}(x)\cdot e^{-\frac{x^2}{2 \sigma^2}}\rbrack(\xi)}}\\
\nonumber
& = &\sqrt{2 \pi}\: \Abs{e^{-\frac{(\sigma \xi)^2}{2}} *  \frac{\sin (w\: \xi)}{\xi}}
= \Abs{\int_{-\infty}^{\infty}e^{-\frac{(\sigma (\xi-\zeta))^2}{2}} \cdot \frac{\sin(w\:\zeta)}{\zeta} d\zeta}\\
&\underset{\sigma \xi = s}{\overset{\sigma\zeta=t}{=}}&
\Abs{\int_{-\infty}^{\infty}e^{-\frac{(s-t)^2}{2}} \cdot \frac{\sin(\sigma^{-1}w t)}{t} dt}\ .
\label{eq:A3}
\eeq
Since $\Abs{\sin(a t)/t} \le \abs{a}$, we obtain
\beqs
(\sigma^{-1}\: \pi) \Abs{\op{F}_1\lbrack \chi_{(-w,w)}(x)\cdot e^{-\frac{x^2}{2 \sigma^2}}\rbrack(\xi)}
&\le &
\sigma^{-1}w \Abs{\int_{-\infty}^{\infty}e^{-\frac{(s-t)^2}{2}} dt} = \sqrt{2 \pi}\sigma^{-1}w\ ,
\eeqs
so in~\eqref{eq:ftlinfnorm} we can take $C_2=\sigma^{-1}w \sqrt{2/\pi} $. For computing $C_1$, let $\delta\in (0,1)$, and recall
$\sigma\: \xi= s$. Without loss of generality assume $\xi>0$. Continuing from~\eqref{eq:A3},
\beqs
 \lefteqn{(\sigma^{-1}\: \pi) \Abs{\op{F}_1\lbrack \chi_{(-w,w)}(x)\cdot e^{-\frac{x^2}{2 \sigma^2}}\rbrack(\xi)} }\\
&\le & 
\Abs{\int_{\abs{s-t}<s\: \delta}e^{-\frac{(s-t)^2}{2}} \cdot \frac{\sin(\sigma^{-1}w t)}{t} dt} + 
\Abs{\int_{\abs{s-t}>s\: \delta}e^{-\frac{(s-t)^2}{2}} \cdot \frac{\sin(\sigma^{-1}w t)}{t} dt}\\
& \le & 
\frac{1}{s (1-\delta)}\int_{\abs{s-t}<s\: \delta}e^{-\frac{(s-t)^2}{2}} dt + 
(\sigma^{-1} w)\int_{\abs{s-t}>s\: \delta}e^{-\frac{(s-t)^2}{2}} dt\\
& \stackrel{s-t=u}{\le} &
\frac{1}{s (1-\delta)} \int_{-\infty}^{\infty}e^{-\frac{u^2}{2}} du + (\sigma^{-1} w)\int_{\abs{u}>s\: \delta}e^{-\frac{u^2}{2}} du\\
&\le& \frac{\sqrt{2 \pi}}{s (1-\delta)} + 2 (\sigma^{-1} w)\int^{\infty}_{s\: \delta}e^{-\frac{s \delta}{2}u}  du
=\frac{\sqrt{2 \pi}}{s (1-\delta)} + 4 \frac{(\sigma^{-1} w)}{s\:\delta}e^{-\frac{(s \delta)^2}{2}}\\
&\le & \left(\sigma \xi\right )^{-1}\left(\frac{\sqrt{2 \pi}}{1-\delta} + \frac{4\sigma^{-1} w}{\delta} \right)\ .
\eeqs
The choice $\delta=1/2$ shows that in~\eqref{eq:ftlinfnorm} we can take
$C_1=2\pi^{-1}\left({\sqrt{2 \pi}} + 4\sigma^{-1} w \right)$.
It is easy to see that for $a_1, a_2, b>0$
\beq
\label{eq:elemineq}
\min\left(\frac{a_1}{\Abs{\xi}},a_2\right)\le \frac{a_1 + a_2 b }{b+\Abs{\xi}}\ .
\eeq
Hence, the inequality~\eqref{eq:ftlinfnorm2} follows from~\eqref{eq:ftlinfnorm} 
by substituting $a_1=C_1$, $a_2 = \sigma C_2$, and $b=1/\sigma$ in~\eqref{eq:elemineq},
with $C_3=C_1 +C_2$.
\end{proof}

The next Lemma shows that $\op{K}^{\sigma, w}$ satisfies Condition~\ref{cond:condsmooth}~{\bf [a]} (recall that the operator is symmetric).
\begin{lemma} There exists a constant $C>0$ depending on the ratio $w /\sigma$ so that 
\label{lma:deblurH1est}
\beq
\label{eq:deblurL2est}
\nnorm{\op{K}^{\sigma, w} u}_{L^2(\R^2)} & \le & C \alpha_w^{-1} \nnorm{u},\ \ \forall u\in L^2(\R^2)\ ,\\
\label{eq:deblurL2estdelta}
\nnorm{\op{K}^{\sigma, w} u}_{L^2(\Omega)} & \le & C \alpha_w^{-1}  \nnorm{u},\ \ \forall u\in L^2(\Omega)\ ,\\
\label{eq:deblurH1est}
\nnorm{\op{K}^{\sigma, w} u}_{H^1(\R^2)} & \le & \alpha_w^{-1}\left(C \max(1,\sigma^{-1})\right) \nnorm{u},\ \ \forall u\in L^2(\R^2)\ ,\\
\label{eq:deblurH1estdelta}
\nnorm{\op{K}^{\sigma, w} u}_{H^1(\Omega)} & \le & \alpha_w^{-1} \left(C \max(1,\sigma^{-1})\right) \nnorm{u},\ \ \forall u\in L^2(\Omega)\ .
\eeq
\end{lemma}
\begin{proof}
Cf.~\cite{MR1157815}, an equivalent $H^k$-norm of a function $v$ on $\R^2$ is given by
$$
\nnorm{v}_{H^k(\R^2)} =  
\nnorm{L_k \cdot \op{F}_2 \lbrack v\rbrack},
$$
where $L_k(\xi) = (1+\abs{\xi}^2)^{\frac{k}{2}}$ for $k\ge 0$. Hence, for $u\in L^2(\R^2)$ and $k=0, 1$
\beq
\nonumber
\alpha_w \nnorm{\op{K}^{\sigma, w} u}_{H^k(\R^2)}& = & \nnorm{L_k \cdot\op{F}_2 \lbrack \left(G_{\sigma}\cdot \chi_{D_w}\right) * u\rbrack}
 =  \nnorm{L_k \cdot\op{F}_2 \lbrack G_{\sigma}\cdot \chi_{D_w}\rbrack \cdot \op{F}_2 \lbrack u\rbrack}\\
\label{eq:plancherel}
& \le & \nnorm{L_k \cdot\op{F}_2 \lbrack G_{\sigma}\cdot \chi_{D_w}\rbrack}_{L^{\infty}(\R^2)} \cdot 
\nnorm{\op{F}_2 \lbrack u\rbrack}\ .
\eeq
By the Plancherel Theorem, $\nnorm{\op{F}_2 \lbrack u\rbrack} = \nnorm{u}$; hence, it remains to estimate the quantity
\beq
\nnorm{L_k \cdot\op{F}_2 \lbrack G_{\sigma}\cdot \chi_{D_w}\rbrack}_{L^{\infty}(\R^2)}.
\eeq
The separability 
$$
(G_{\sigma}\cdot  \chi_{D_w})(x_1,x_2) =  \sigma^{-2}
\left(e^{-\frac{{x_1}^2}{2\sigma^2}} \chi_{(-w,w)}(x_1)\right) \cdot \left(e^{-\frac{{x_2}^2}{2\sigma^2}} \chi_{(-w,w)}(x_2)\right)
$$
implies that
\beqs
\op{F}_2 \lbrack G_{\sigma}\cdot \chi_{D_w}\rbrack (\xi_1, \xi_2)& = 
\sigma^{-2} \op{F}_1 \lbrack e^{-\frac{{x_1}^2}{2\sigma^2}} \chi_{(-w,w)}(x_1) \rbrack(\xi_1) \cdot 
\op{F}_1 \lbrack e^{-\frac{{x_2}^2}{2\sigma^2}} \chi_{(-w,w)}(x_2) \rbrack(\xi_2) .
\eeqs
The case $k=0$ is easy, since by~\eqref{eq:ftlinfnorm2} 
\beqs
\nnorm{L_0\cdot\op{F}_2 \lbrack G_{\sigma}\cdot \chi_{D_w}\rbrack}_{L^{\infty}(\R^2)} =
\nnorm{\op{F}_2 \lbrack G_{\sigma}\cdot \chi_{D_w}\rbrack}_{L^{\infty}(\R^2)} 
\le (C_3)^2\ ,
\eeqs
which, in light of~\eqref{eq:plancherel},  proves~\eqref{eq:deblurL2est}, and hence~\eqref{eq:deblurL2estdelta}. 

For $k=1$, we have
\beqs
\Abs{\left(L_1 \cdot\op{F}_2 \lbrack G_{\sigma}\cdot \chi_{D_w}\right)(\xi_1,\xi_2)\rbrack}& = &
\Abs{(1+\xi_1^2+\xi_2^2)^{\frac{1}{2}}\op{F}_2 \lbrack G_{\sigma}\cdot \chi_{D_w}\rbrack (\xi_1, \xi_2)}\\
& \stackrel{\eqref{eq:ftlinfnorm2}}{\le} & (C_3)^2\frac{\sqrt{1+\xi_1^2+\xi_2^2}}{(1+\sigma\abs{\xi_1})(1+\sigma \abs{\xi_2})}\\
&\le & (C_3)^2\max(1,\sigma^{-1})\ .
\eeqs
The latter inequality follows from 
\beq
\label{eq:Asigineq}
\frac{\sqrt{A+t^2}}{1+\sigma t}\le \max(\sqrt{A},\sigma^{-1}),\ \ \forall \ A, \sigma >0,\ t\ge 0\ .
\eeq
Namely, for all $\xi_1, \xi_2\in \R$,
\beqs
\frac{\sqrt{1+\xi_1^2+\xi_2^2}}{(1+\sigma\abs{\xi_1})(1+\sigma \abs{\xi_2})} &\stackrel{\eqref{eq:Asigineq}}{\le}&
\frac{1}{1+\sigma \abs{\xi_2}}\max\left( \sqrt{1+\xi_2^2},\sigma^{-1}\right)\\
&\stackrel{\eqref{eq:Asigineq}}{\le}& 
\max\left(\max\left(1,\sigma^{-1} \right), \frac{\sigma^{-1}}{1+\sigma \abs{\xi_2}} \right) = \max(1,\sigma^{-1})\ .
\eeqs
This proves that
\beq
\nnorm{L_1 \cdot\op{F}_2 \lbrack G_{\sigma}\cdot \chi_{D_w}\rbrack}_{L^{\infty}(\R^2)} \le (C_3)^2 \max(1,\sigma^{-1}) ,
\eeq
thus showing that~\eqref{eq:deblurH1est} holds with $C=(C_3)^2$.

Given $u\in L^2(\Omega)$, we consider its extension (still denoted $u$) with $0$ outside $\Omega$, and we 
apply the inequality~\eqref{eq:deblurH1est} to obtain
\beqs
\nnorm{\op{K}^{\sigma, w} u}_{H^1(\Omega)} &\le& \nnorm{\op{K}^{\sigma, w} u}_{H^1(\R^2)} \\ 
&\le& 
 \alpha_w^{-1}\left(C \max(1,\sigma^{-1})\right) \nnorm{u}_{L^2(\R^2)} = \alpha_w^{-1}\left(C \max(1,\sigma^{-1})\right) \nnorm{u}_{L^2(\Omega)}\ ,
\eeqs
which proves~\eqref{eq:deblurH1estdelta}.
\end{proof}
\begin{lemma}
\label{lma:wsensitivity}
If $0<w_0\le w_1<w_2<1/2$, then
\beq
\label{eq:wsensitivity1}
\nnorm{(\op{K}^{\sigma, w_1} - \op{K}^{\sigma, w_2})u}_{L^2(\R^2)} &\le & C\: (w_2-w_1) \nnorm{u}_{L^2(\R^2)}\ ,\\
\label{eq:wsensitivity2}
\nnorm{(\op{K}^{\sigma, w_1} - \op{K}^{\sigma, w_2})u}_{L^2(\Omega)} & \le & C \: (w_2-w_1) \nnorm{u}_{L^2(\Omega)}\ ,
\eeq
where the constant $C$ only depends on $\sigma$ and $w_0$.
\end{lemma}
\begin{proof} Since $w_0\le w_1<w_2$, it follows that $D_{w_0}\subseteq D_{w_1}\subset D_{w_2}$.
Let $D_{w_1,w_2}=D_{w_2}\setminus D_{w_1}$. Because $w_1+w_2<1$,  
$$\mu(D_{w_1,w_2}) = 4(w_2^2-w_1^2) < 4(w_2-w_1)\ .$$ 
Also,
\beqs
\alpha_{w_2}-\alpha_{w_1} = \frac{1}{2 \pi}\int_{D_{w_1,w_2}} G_{\sigma}(x) dx \le \frac{2(w_2-w_1)}{\pi}\:\norm{G_{\sigma}}_{L^{\infty}(\R^2)} = 
\frac{2(w_2-w_1)}{\sigma^2\pi}\ .
\eeqs
Therefore
\beqs
\alpha_{w_1}^{-1}-\alpha_{w_2}^{-1} = \frac{\alpha_{w_2}-\alpha_{w_1}}{\alpha_{w_1}\:\alpha_{w_2}} \le \frac{2(w_2-w_1)}{\alpha_{w_0}^2 \sigma^2\pi} 
\eqdef {c_1}(w_2-w_1)\ ,
\eeqs
since $\alpha_{w_2}>\alpha_{w_1}\ge \alpha_{w_0}$. For $u\in L^2(\R^2)$
\beq
\nonumber
\lefteqn{\nnorm{(\op{K}^{\sigma, w_2} - \op{K}^{\sigma, w_1})u}_{L^2(\R^2)}}\\\nonumber
 & = & \nnorm{\left(\alpha_{w_2}^{-1}G_{\sigma}(\chi_{D_{w_2}} - \chi_{D_{w_1}}) + 
(\alpha_{w_2}^{-1}-\alpha_{w_1}^{-1}) G_{\sigma}\cdot\chi_{D_{w_1}} \right) * u}_{L^2(\R^2)} \\\nonumber
&\le & \alpha_{w_2}^{-1}\nnorm{(G_{\sigma}\cdot\chi_{D_{w_1,w_2}}) * u}_{L^2(\R^2)}+ 
\abs{\alpha_{w_2}^{-1}-\alpha_{w_1}^{-1}}\cdot \nnorm{(G_{\sigma}\cdot\chi_{D_{w_1}}) * u}_{L^2(\R^2)}\\
\label{eq:youngneeded}
&\le & \left(\alpha_{w_2}^{-1}\nnorm{G_{\sigma}\cdot\chi_{D_{w_1,w_2}}}_{L^1(\R^2)} + 
\abs{\alpha_{w_2}^{-1}-\alpha_{w_1}^{-1}}\cdot \nnorm{G_{\sigma}\cdot\chi_{D_{w_1}}}_{L^1(\R^2)} \right)\nnorm{u}_{L^2(\R^2)} \hspace{30pt}\\\nonumber
&\le& 
\nnorm{G_{\sigma}}_{L^{\infty}(\R^2)}\left(\alpha_{w_2}^{-1}\cdot \mu(D_{w_1,w_2}) +c_1(w_2-w_1)\mu(D_{w_1})\right)
\cdot  \nnorm{u}_{L^2(\R^2)} \\\nonumber
&\le & \sigma^{-2} \left(4\alpha_{w_0}^{-1} + c_1\right) (w_2-w_1) \nnorm{u}_{L^2(\R^2)} \eqdef C (w_2-w_1) \nnorm{u}_{L^2(\R^2)}\ ,
\eeq
where in~\eqref{eq:youngneeded} we used Young's inequality for convolutions. Naturally, $C$ only depends on $\sigma$ and $w_0$.
As before, the inequality~\eqref{eq:wsensitivity2} follows from~\eqref{eq:wsensitivity1} by extending $u\in L^2(\Omega)$ with zero
outside $\Omega$.
\end{proof}

\subsection{The discretization of $\op{K}^{\sigma, w}$ and convergence estimates}
\label{ssec:discrete_blurr}
Recall that the domain $\Omega$ is partitioned 
uniformly into $N_h=n^2$ 
squares $\Omega=\cup_{k=1}^{N_h} R_k$ with \mbox{$h=1/n$}, and let $z_k$ be the center of the square $R_k$. We denote by~$\op{U}_h$ the
space of piecewise constant functions on $\Omega$ with respect to the aforementioned partition, with functions
in  $\op{U}_h$ being determined by their values at the nodes $z_k$. For this example we take~$\op{V}_h=\op{U}_h$, so
$\op{K}_h\in \mathfrak{L}(\op{U}_h)$. Note that $\op{V}_{h}\not\subset H^1({\Omega})$. For convenience and consistency with the continuous
case we extend the grid to $\R^2$ and we extend any function in  $\op{U}_h$ with zero outside $\Omega$. In this section $\nnorm{\cdot}$ 
denotes the $L^2$-norm on $\Omega$.

The first step towards discretization is to slightly enlarge the domain of integration in~\eqref{eq:blurdef}, when  $x=z_k$, to be a union of elements 
in the partition.
Hence, for a given node $z_k$,  denote by $\op{N}_k$ the set of indices $l$ for which  Int$(R_k)$   intersects
the ball $\op{B}_w(z_k) = \{y\in \Omega\: :\: \abs{y-z_k}_{\infty}< w\}$. It is easy to see that 
$$
\op{N}_k=\left\{l\: :\: \abs{z_k-z_l}_{\infty}\le w (1+h/2)\right \}\ .
$$
So for $x=z_k$ the domain of integration in~\eqref{eq:blurdef} becomes
$$
\bigcup_{l\in \op{N}_k} R_l = \op{B}_{w_h}(z_k)\ ,\ \ \ \mathrm{with}\ \ 
w_h\eqdef \left(\left\lceil\frac{w}{h}-\frac{1}{2}\right\rceil +\frac{1}{2}\right) h .
$$
Essentially, this is the smallest ball (in the $\abs{\: \cdot\:}_{\infty}$-norm) 
centered at $z_k$ that includes $\op{B}_w(z_k)$ and is also a union of mesh-elements.
Note that
\beq
\label{eq:wdiff}
w\le w_h < w+h\ .
\eeq
The discretization 
$\op{K}_h\in \mathfrak{L}(\op{U}_h)$ of $\op{K}=\op{K}^{\sigma, w}$ is given by
\beq
\label{eq:discrete_deblurr}
({\op{K}}_h u) (z_k) \eqdef (1+\eta)\:\sum_{l\in \op{N}_k} {\gamma}_{k l}\: u(z_l) \eqdef (1+\eta)(\widetilde{K}_h u)(z_k)
\eeq
with ${\gamma}_{k l} = h^2 (2\alpha_{w_h} \pi)^{-1} G_{\sigma}(z_k-z_l)$, where $\eta$ is chosen so that
\beq
\label{eq:discrete_deblurr_cond}
(1+\eta)\sum_{l\in \op{N}_k}{\gamma}_{k l} = 1\ ,
\eeq
for all $1\le k \le N_h$ for which $\op{B}_{w_h}(z_k)\subseteq \Omega$. 
Note that due to the uniformity of the grid, ${\gamma}_{k l}$ only depends on the vector $(z_k-z_l)$.
The next result shows that the operators $\op{K}^{\sigma,w}$ and ${\op{K}}_h$ satifsy Condition~\ref{cond:condsmooth}{\bf [b]} and {\bf [c]}.
\begin{theorem}
\label{lma:approx_deblur}
There exists a constant $C>0$ which depends on $\sigma, w$  so that 
\beq
\label{eq:approx_deblur}
\nnorm{(\op{K}^{\sigma,w} -{\op{K}}_h)u} &\le & C h \nnorm{u},\ \ \forall u\in \op{U}_h\\
\label{eq:linftybd_deblur}
\nnorm{{\op{K}}_hu}_{L^{\infty}(\Omega)} & \le & C \nnorm{u},\ \ \forall u\in \op{U}_h\ ,
\eeq
assuming $h$ is sufficiently small.
\end{theorem}
\begin{proof}
Throughout this analysis $C$ denotes a generic positive constant depending on $\sigma, w$ but not on $h$.
By Lemma~\ref{lma:wsensitivity} 
\beq
\label{eq:senswh}
\nnorm{(\op{K}^{\sigma,w} - \op{K}^{\sigma,w_h})u} \le C (w_h-w)\nnorm{u} \stackrel{\eqref{eq:wdiff}}{\le} C h \nnorm{u},\ \ \forall u\in \op{U}_h\ .
\eeq
Let $\op{I}_h: L^2(\Omega)\to \op{U}_h$ be the interpolation operator
$$(\op{I}_h v) (z_k) = \frac{1}{\mu(R_k)}\int_{R_k} v(x) dx\ .$$
By the Bramble-Hilbert Lemma and Lemma~\ref{lma:deblurH1est}
\beq
\label{eq:BHlemma}
\nnorm{(\op{K}^{\sigma,w_h} - \op{I}_h\op{K}^{\sigma,w_h})u} \le C h \nnorm{\op{K}^{\sigma,w_h} u}_1 \le C h \nnorm{u},\ \ \forall u\in \op{U}_h\ ,
\eeq
with $C$ depending only on $w_h/\sigma$ and the domain $\Omega$; hence $C$ can be bounded uniformly with respect to $h$.
We now fix an index $1\le k \le N_h$, and let $l\in \op{N}_k$. The choice ${\gamma}_{k l}$ is so that $(\widetilde{K}_h u)(z_k)$
is obtained by replacing in~\eqref{eq:blurdef} (with $w_h$ instead of $w$ and $x=z_k$) the integral on each $R_l$ ($l\in \op{N}_k$)
by the midpoint cubature. Therefore, since $u$ is constant on each $R_l$,
\beq
\label{eq:estkminusktilde}
(\op{K}^{\sigma,w_h}-\widetilde{K}_h)u (z_k)  = \frac{\alpha_{w_h}^{-1}}{2\pi}
\sum_{l\in \op{N}_k}\int_{R_l} \left(G_{\sigma}(z_k-y)-G_{\sigma}(z_k-z_l)\right)u(z_l) dy.\hspace{30pt}
\eeq
Let $M_{\sigma}$ be the upper bound of the second order (bilinear) Fr{\' e}chet differential of $G_{\sigma}$, regarded as a function from 
$(\R^2,\abs{\:\cdot\:}_{\infty})$ to $\R$ (it is easy to see that all differentials of $G_{\sigma}$ are bounded uniformly on $\R^2$).
Then the Taylor expansion of the function $y\mapsto G_{\sigma}(z_k-y)$ around $z_l$ gives
$$
\abs{G_{\sigma}(z_k-y)-G_{\sigma}(z_k-z_l) + dG_{\sigma}(z_k) (y-z_l)}_{\infty} \le \frac{M_{\sigma}}{2} \abs{y-z_l}^2\ .
$$
Due to the symmetry of $R_l$  with respect to $z_l$ we have 
$$
\int_{R_l}dG_{\sigma}(z_k) (y-z_l) dy = 0\ .
$$
Hence
\beq
\label{eq:esttayint}
\Abs{\int_{R_l} \left(G_{\sigma}(z_k-y)-G_{\sigma}(z_k-z_l)\right) dy}\le \mu(R_l)\frac{M_{\sigma}h^2}{2} \ .
\eeq
By~\eqref{eq:estkminusktilde} and~\eqref{eq:esttayint} 
\beq
\label{eq:pointwiseest}
\Abs{(\op{K}^{\sigma,w_h}-\widetilde{K}_h)u (z_k)} \le h^2\frac{\alpha_{w_h}^{-1}M_{\sigma} }{4\pi}
\sum_{l\in \op{N}_k}\mu(R_l) \abs{u(z_l)} \le h^2 \frac{\alpha_{w}^{-1}M_{\sigma} }{4\pi} \nnorm{u}_{L^1(\Omega)}. \hspace{30pt}
\eeq
Using the continuous inclusions $L^{\infty} \subset L^{2} \subset L^{1}$
\beq
\label{eq:l2inf1}
\nnorm{(\op{I}_h\op{K}^{\sigma,w_h}-\widetilde{K}_h)u}\le
  C \nnorm{(\op{I}_h\op{K}^{\sigma,w_h}-\widetilde{K}_h)u}_{L^{\infty}(\Omega)} \stackrel{\eqref{eq:pointwiseest}}{\le} 
  C h^2 \nnorm{u}_{L^1(\Omega)}\le C h^2 \nnorm{u}.\hspace{30pt}
\eeq
The estimates~\eqref{eq:senswh},~\eqref{eq:BHlemma}, and~\eqref{eq:l2inf1}
imply that
\beq
\label{eq:approx_Ktilde}
\nnorm{(\op{K}^{\sigma,w} -{\widetilde{\op{K}}}_h)u} \le C h \nnorm{u},\ \ \forall u\in \op{U}_h\ .
\eeq
Using~\eqref{eq:deblurL2estdelta} and~\eqref{eq:approx_Ktilde}, we also obtain the uniform estimate
\beq
\label{eq:Ktildeunif}
\nnorm{\widetilde{\op{K}}_h u} \le C \nnorm{u},\ \ \forall u\in \op{U}_h\ .
\eeq

For the final step, recall that $\op{K}_h =(1+\eta){\widetilde{\op{K}}}_h$, with $\eta$ chosen to satisfy~\eqref{eq:discrete_deblurr_cond}.
To estimate $\eta$, let $z_k$ be so that $\op{B}_{w_h}(z_k)\subseteq \Omega$ and  $u\equiv 1\in \op{U}_h$. 
By definition of the coefficients $\gamma_{k l}$ (which are all positive)  
\beqs
\sum_{l\in \op{N}_k}{\gamma}_{k l} = (\widetilde{K}_h u )(z_k) 
\stackrel{\eqref{eq:pointwiseest}}{\ge} (\op{K}^{\sigma,w_h}u)(z_k)-C h^2 \nnorm{u}_{L^1(\Omega)} =1-C h^2 \ge\frac{1}{2},
\eeqs
assuming $h$ is sufficiently small. Hence
\beq
\label{eq:deltah2est}
\frac{1}{2} \abs{\eta} \le \abs{\eta}\sum_{l\in \op{N}_k}{\gamma}_{k l} \stackrel{\eqref{eq:discrete_deblurr_cond}}{=} 
\abs{1- \sum_{l\in \op{N}_k}{\gamma}_{k l}} = \Abs{\left( \op{K}^{\sigma,w_h}-{\widetilde{\op{K}}}_h \right)u (z_k)} 
\stackrel{\eqref{eq:pointwiseest}}{\le} C h^2.\hspace{30pt}
\eeq
Therefore 
\beq
\label{eq:approx_KtildeKh}
\nnorm{(\op{K}_h -{\widetilde{\op{K}}}_h)u} = \abs{\eta} \nnorm{{\widetilde{\op{K}}}_hu}\stackrel{\eqref{eq:Ktildeunif},\eqref{eq:deltah2est}}{\le} 
C h^2 \nnorm{u},\ \ \forall u\in \op{U}_h\ .
\eeq
The conclusion follows from~\eqref{eq:approx_Ktilde} and~\eqref{eq:approx_KtildeKh}.

Given a node $z_k$ we have
\beqs
\abs{(\op{K}^{\sigma, w_h} u) (z_k)} \le \frac{\alpha_{w_h}^{-1}}{2\pi}\int_{\abs{z_k-y}_{\infty}<w_h} G_{\sigma}(z_k-y) \abs{u(y)} dy
\le \frac{\alpha_{w_h}^{-1}}{2\sigma^2 \pi} \nnorm{u}_{L^1(\Omega)} \le C \nnorm{u}_{L^1(\Omega)}\ .
\eeqs
By~\eqref{eq:pointwiseest}
\beqs
\nnorm{\widetilde{\op{K}}_h u}_{L^{\infty}(\Omega)}\le  \nnorm{\op{K}^{\sigma, w_h} u}_{L^{\infty}(\Omega)}+
\nnorm{(\op{K}^{\sigma, w_h} - \widetilde{\op{K}}_h )u}_{L^{\infty}(\Omega)}\le C \nnorm{u}_{L^1(\Omega)} \le C \nnorm{u}\ ,
\eeqs
and~\eqref{eq:linftybd_deblur} now follows from~\eqref{eq:deltah2est} and $\op{K}_h =(1+\eta){\widetilde{\op{K}}}_h$.
\end{proof}

%% file: main.bbl
\begin{thebibliography}{10}

\bibitem{MR2891920}
{\sc Sven Beuchler, Clemens Pechstein, and Daniel Wachsmuth}, {\em Boundary
  concentrated finite elements for optimal boundary control problems of
  elliptic {PDE}s}, Comput. Optim. Appl., 51 (2012), pp.~883--908.

\bibitem{MR2421947}
{\sc George Biros and G{\"u}nay Do{\v{g}}an}, {\em A multilevel algorithm for
  inverse problems with elliptic {PDE} constraints}, Inverse Problems, 24
  (2008), pp.~034010, 18.

\bibitem{MR2160699}
{\sc A.~Borz{\`{\i}} and K.~Kunisch}, {\em A multigrid scheme for elliptic
  constrained optimal control problems}, Comput. Optim. Appl., 31 (2005),
  pp.~309--333.

\bibitem{MR2505585}
{\sc Alfio Borzi and Volker Schulz}, {\em Multigrid methods for {PDE}
  optimization}, SIAM Rev., 51 (2009), pp.~361--395.

\bibitem{MR2322235}
{\sc Dietrich Braess}, {\em Finite elements}, Cambridge University Press,
  Cambridge, third~ed., 2007.
\newblock Theory, fast solvers, and applications in elasticity theory,
  Translated from the German by Larry L. Schumaker.

\bibitem{MR2373954}
{\sc Susanne~C. Brenner and L.~Ridgway Scott}, {\em The mathematical theory of
  finite element methods}, vol.~15 of Texts in Applied Mathematics, Springer,
  New York, third~ed., 2008.

\bibitem{doi:10.1080/10556788.2013.854356}
{\sc Andrei Dr{\u{a}}g{\u{a}}nescu}, {\em Multigrid preconditioning of linear
  systems for semi-smooth {N}ewton methods applied to optimization problems
  constrained by smoothing operators}, Optim. Methods Softw., 29 (2014),
  pp.~786--818.

\bibitem{MR2429872}
{\sc Andrei Dr{\u{a}}g{\u{a}}nescu and Todd~F. Dupont}, {\em Optimal order
  multilevel preconditioners for regularized ill-posed problems}, Math. Comp.,
  77 (2008), pp.~2001--2038.

\bibitem{Dra:Soa:mgstokes}
{\sc Andrei Dr{\u{a}}g{\u{a}}nescu and Ana~Maria Soane}, {\em Multigrid
  solution of a distributed optimal control problem constrained by the {S}tokes
  equations}, Appl. Math. Comput., 219 (2013), pp.~5622--5634.

\bibitem{Dra:Pet:ipm}
{\sc Andrei Dr\u{a}g\u{a}nescu and Cosmin Petra}, {\em Multigrid
  preconditioning of linear systems for interior point methods applied to a
  class of box-constrained optimal control problems}, SIAM Journal on Numerical
  Analysis, 50 (2012), pp.~328--353.

\bibitem{MR2001h:65069}
{\sc Martin Hanke and Curtis~R. Vogel}, {\em Two-level preconditioners for
  regularized inverse problems. {I}. {T}heory}, Numer. Math., 83 (1999),
  pp.~385--402.

\bibitem{MR2740620}
{\sc Roland Herzog and Ekkehard Sachs}, {\em Preconditioned conjugate gradient
  method for optimal control problems with control and state constraints}, SIAM
  J. Matrix Anal. Appl., 31 (2010), pp.~2291--2317.

\bibitem{MR1972219}
{\sc M.~Hinterm{\"u}ller, K.~Ito, and K.~Kunisch}, {\em The primal-dual active
  set strategy as a semismooth {N}ewton method}, SIAM J. Optim., 13 (2002),
  pp.~865--888 (electronic) (2003).

\bibitem{MR2085262}
{\sc Michael Hinterm{\"u}ller and Michael Ulbrich}, {\em A mesh-independence
  result for semismooth {N}ewton methods}, Math. Program., 101 (2004),
  pp.~151--184.

\bibitem{MR1276702}
{\sc R.~H.~W. Hoppe and R.~Kornhuber}, {\em Adaptive multilevel methods for
  obstacle problems}, SIAM J. Numer. Anal., 31 (1994), pp.~301--323.

\bibitem{MR1986801}
{\sc Barbara Kaltenbacher}, {\em V-cycle convergence of some multigrid methods
  for ill-posed problems}, Math. Comp., 72 (2003), pp.~1711--1730 (electronic).

\bibitem{MR1151773}
{\sc J.~Thomas King}, {\em Multilevel algorithms for ill-posed problems},
  Numer. Math., 61 (1992), pp.~311--334.

\bibitem{MR2491821}
{\sc O.~Lass, M.~Vallejos, A.~Borzi, and C.~C. Douglas}, {\em Implementation
  and analysis of multigrid schemes with finite elements for elliptic optimal
  control problems}, Computing, 84 (2009), pp.~27--48.

\bibitem{doi:10.1137/140975711}
{\sc Margherita Porcelli, Valeria Simoncini, and Mattia Tani}, {\em
  Preconditioning of active-set newton methods for pde-constrained optimal
  control problems}, SIAM Journal on Scientific Computing, 37 (2015),
  pp.~S472--S502.

\bibitem{MR97k:65299}
{\sc Andreas Rieder}, {\em A wavelet multilevel method for ill-posed problems
  stabilized by {T}ikhonov regularization}, Numer. Math., 75 (1997),
  pp.~501--522.

\bibitem{MR1157815}
{\sc Walter Rudin}, {\em Functional analysis}, International Series in Pure and
  Applied Mathematics, McGraw-Hill, Inc., New York, second~ed., 1991.

\bibitem{MR2831057}
{\sc Joachim Sch{\"o}berl, Ren{\'e} Simon, and Walter Zulehner}, {\em A robust
  multigrid method for elliptic optimal control problems}, SIAM J. Numer.
  Anal., 49 (2011), pp.~1482--1503.

\bibitem{MR2863635}
{\sc Stefan Takacs and Walter Zulehner}, {\em Convergence analysis of multigrid
  methods with collective point smoothers for optimal control problems},
  Comput. Vis. Sci., 14 (2011), pp.~131--141.

\bibitem{MR3069094}
\leavevmode\vrule height 2pt depth -1.6pt width 23pt, {\em Convergence analysis
  of all-at-once multigrid methods for elliptic control problems under partial
  elliptic regularity}, SIAM J. Numer. Anal., 51 (2013), pp.~1853--1874.

\bibitem{MR2839219}
{\sc Michael Ulbrich}, {\em Semismooth {N}ewton methods for variational
  inequalities and constrained optimization problems in function spaces},
  vol.~11 of MOS-SIAM Series on Optimization, Society for Industrial and
  Applied Mathematics (SIAM), Philadelphia, PA, 2011.

\end{thebibliography}
